\documentclass{ipiOF}
\usepackage{amsmath}
  \usepackage{paralist}
  \usepackage{graphics} %% add this and next lines if pictures should be in esp format
  \usepackage{epsfig} %For pictures: screened artwork should be set up with an 85 or 100 line screen
\usepackage{graphicx}  \usepackage{epstopdf}%This is to transfer .eps figure to .pdf figure; please compile your paper using PDFLeTex or PDFTeXify.
 \usepackage[colorlinks=true]{hyperref}
 \usepackage{comment}
 \usepackage{enumitem}
\usepackage{autonum}
\usepackage{subcaption}
\usepackage{makebox}
\usepackage{booktabs}
\hypersetup{urlcolor=blue, citecolor=red}

\def\currentvolume{X}
 \def\currentissue{X}
  \def\currentyear{200X}
   
    \def\ppages{X--XX}
     \def\DOI{10.3934/ipi.2022002}

\newtheorem{theorem}{Theorem}[section]

\newtheorem{lemma}[theorem]{Lemma}
\newtheorem{proposition}{Proposition}

\newtheorem{assumption}{Assumption}
\theoremstyle{definition}
\newtheorem{method}{Method}

\title[Direct regularized reconstruction in 3D EIT] %Use the shortened version of the full title
      {Direct regularized reconstruction for the three-dimensional Calder\'on problem}

\author[Kim Knudsen and Aksel Kaastrup Rasmussen]{}


 \email{kiknu@dtu.dk}
 \email{akara@dtu.dk}

\thanks{$^*$ Corresponding author}

\begin{document}
\maketitle

% Enter the first author's name and address:
\centerline{\scshape Kim Knudsen and Aksel Kaastrup Rasmussen$^*$}
\medskip
{\footnotesize
% please put the address of the first author
 \centerline{Technical University of Denmark}
   \centerline{Department of Applied Mathematics and Computer Science}
   \centerline{DK-2800 Kgs. Lyngby, Denmark}
} % Do not forget to end the {\footnotesize by the sign }
%\medskip
%\centerline{\scshape Aksel Kaastrup Rasmussen$^*$}
%\medskip
%{\footnotesize
% % please put the address of the second  and third author
% \centerline{Technical University of Denmark}
%   \centerline{Department of Applied Mathematics and Computer Science}
%   \centerline{DK-2800 Kgs. Lyngby, Denmark}
%}

\bigskip
% The name of the associate editor will be entered by an editorial staff
% "Communicated by the associate editor name" is not needed for special issue.
%\centerline{(Communicated by Mikko Salo)}

%The abstract of your paper
\begin{abstract}
Electrical Impedance Tomography gives rise to the severely ill-posed
Calderón problem of determining the electrical conductivity distribution
in a bounded domain from knowledge of the associated Dirichlet-to-Neumann
map for the governing equation. The uniqueness and stability questions
for the three-dimensional problem were largely answered in the affirmative in the 1980’s using complex
geometrical optics solutions, and this led further to a direct reconstruction method relying on a non-physical scattering transform. In this paper, the reconstruction problem is taken one step further towards practical applications by considering data contaminated by noise. Indeed, a regularization strategy for the three-dimensional Calderón problem is presented based on a suitable and explicit truncation of the scattering transform. This gives a certified, stable and direct reconstruction method that is robust to small perturbations of the data. Numerical tests on simulated noisy data illustrate the feasibility and regularizing effect of the method, and suggest that the numerical implementation performs better than predicted by theory.
\end{abstract}

%The title of your section 1
\section{Introduction}
Electrical Impedance Tomography (EIT) provides a noninvasive method of obtaining information on the electrical conductivity distribution of electric conductive media from exterior electrostatic measurements of currents and voltages. There are many applications in medical imaging including early detection of breast cancer \cite{cherepenin2002,Zou200379}, hemorrhagic stroke detection \cite{malone2014a,goren2018a}, pulmonary function monitoring \cite{Adler19971762,frerichs2019a,Leonhardt20121917} and targeting control in transcranial brain stimulation \cite{schmidt2015a}. Applications also include industrial testing, for example, crack damage detection in cementitious structures \cite{Hou2009ElectricalIT,Hallaji2014}, and subsurface geophysical imaging \cite{zhao2013a}. The mathematical problem of EIT is called the Calderón problem and was first formulated by A.P. Calderón in 1980 \cite{calderoninverse} as follows:
Consider a bounded Lipschitz domain $\Omega \subset \mathbb{R}^3$ filled with a conductor with a distribution $\gamma \in L^\infty(\Omega)$, $\gamma\geq c >0$. Under the assumption of no sinks or sources of current in the domain, applying an electrical  surface potential $f \in H^{1/2}(\partial \Omega)$ induces an electrical potential $u \in H^1(\Omega)$, which uniquely solves the conductivity equation
\begin{equation} \label{eq:cond}
\begin{aligned}
\nabla \cdot (\gamma \nabla u) &=0 & & \text { in } \Omega, \\
u &=f & & \text { on } \partial \Omega.
\end{aligned}
\end{equation}
The Dirichlet-to-Neumann map $\Lambda_{\gamma}: H^{1/2}(\partial \Omega) \rightarrow H^{-1/2}(\partial \Omega)$ is defined as
\begin{equation}
\Lambda_{\gamma} f=\gamma \partial_{\nu} u|_{\partial \Omega},
\end{equation}
and associates a voltage potential on the boundary with a corresponding normal current flux. All pairs $(f,\gamma \partial_{\nu} u|_{\partial \Omega})$, or equivalently the Dirichlet-to-Neumann map, constitute the available {data}.

 The forward problem is the problem of determining the Dirichlet-to-Neumann map given the conductivity, and it amounts to solving the boundary value problem \eqref{eq:cond} for all possible $f$. The Calderón problem now asks whether $\gamma$ is uniquely determined by $\Lambda_\gamma$, and how to stably reconstruct $\gamma$ from $\Lambda_\gamma$, if possible.
Uniqueness and reconstruction were considered and solved for sufficiently regular conductivity distributions in dimension $n\geq 3$ in a series of papers \cite{Nachman1988595, nachman1988a,  Novikov1988263, sylvester1987a, carorogers2016}. The results are based on complex geometrical optics (CGO) solutions to a Schrödinger equation derived from \eqref{eq:cond}. The first step of the reconstruction method is to recover the CGO solutions on $\partial \Omega$ by solving a weakly singular boundary integral equation with an exponentially growing kernel.  The second step is obtaining the so-called non-physical scattering transform, which approximates in a large complex frequency limit the Fourier transform of $\gamma^{-1/2}\Delta \gamma^{1/2}$. Applying the inverse Fourier transform and solving a boundary value problem yields $\gamma$ in the third step. Numerical algorithms following the scattering transform approach in dimension $n\geq 3$ have been developed by approximating the scattering transform \cite{bikowski2011a, knudsen2011a,hamilton2020a,boverman2009a}, by approximating the boundary integral equations \cite{delbary2012a}, and for the full theoretical reconstruction algorithm \cite{delbary2014a}. A reconstruction algorithm for conductivity distributions close to a constant has been suggested, but not implemented \cite{cornean2006a}.

A similar scattering transform approach combined with tools from complex analysis enables uniqueness and reconstruction \cite{nachman1996a} for the two-dimensional Calderón problem. More recently, a final affirmative answer was given to the question of uniqueness for a general bounded conductivity distribution in two dimensions \cite{astala2006a}. Numerical algorithms and implementation for the two-dimensional problem have been considered \cite{knudsen2003a, knudsen2007a, mueller2003a,  mueller2002a, siltanen2001a, siltanen2000a} and a regularization analysis and full implementation was given in \cite{Knudsen2009a}. We stress that in any practical case the Calderón problem is three-dimensional, since applying potentials on the boundary of a planar cross section of $\Omega$ leads to current flow leaving the plane.

The Calderón problem is known to be severely ill posed. Conditional stability estimates exist \cite{alessandrini1988a,alessandrini1990a} of the form
\begin{equation}\label{eq:stability}
 		\|\gamma_1-\gamma_2\|_{L^\infty(\Omega)} \leq f(\|\Lambda_{\gamma_1}-\Lambda_{\gamma_2}\|_{Z}),
\end{equation}
for an appropriate function space $Z$ and continuous function $f$ with $f(0)=0$ of logarithmic type. Furthermore, logarithmic stability is optimal \cite{mandache2001a}. While this is relevant for the theoretical reconstruction, there is no guarantee that a practically measured $\Lambda_\gamma^\varepsilon$ of a perturbed Dirichlet-to-Neumann map  is the Dirichlet-to-Neumann map of any conductivity. We emphasize that in any practical case we can not have infinite-precision data, but rather a noisy finite approximation. Consequently, any computational algorithm for the problem needs regularization.

Classical regularization theory for inverse problems is given in \cite{engl1996a,kirsch1996a} with a focus on least squares formulations.
A common approach to regularization for the Calderón problem is based on iterative regularized least-squares, and convergence of such methods is analyzed in \cite{dobson1992, rondi2008a, rondi2016a, lechleiter2008a, jin2012a} in the context of EIT.
A quantitative comparison of CGO-based methods and iterative regularized methods is given in \cite{hamilton2020a}. Reconstruction by statistical inversion is developed in \cite{kaipio2000a, dunlop2016a}, where in the latter, the problem is posed in an infinite-dimensional Bayesian framework. A different statistical approach to the Calderón problem shows stable reconstruction of the surface conductivity on a domain given noisy data \cite{caro2017}. Convergence estimates in probability of a statistical estimator (posterior mean) to the true conductivity given noisy data with a sufficiently small noise level are considered in \cite{kweku2019}.

In this paper we provide a direct CGO-based regularization strategy with an admissible parameter choice rule for reconstruction in the three-dimensional Calderón problem under the following assumptions:

%%%% Assumptions %%%%%

\begin{assumption}
For simplicity of exposition, we assume the domain of interest $\Omega$ is embedded in the unit ball in $\mathbb{R}^3$. Furthermore, we assume $\partial \Omega$ is smooth.
\end{assumption}
\begin{assumption}[Parameter and data space]\label{assumption2}
	We consider the forward map $F:D(F)\subset L^\infty(\Omega) \rightarrow Y$, $\gamma\mapsto \Lambda_\gamma$ with the following definition of $D(F)$.
	Let $\Pi>0$ and $0<\rho<1$, then $\gamma \in D(F)\subset L^\infty(\Omega)$ satisfies
\begin{equation}\label{Df}
	\begin{aligned}
		\|\gamma\|_{C^2(\overline{\Omega})}&\leq \Pi,\\
		\gamma(x) &\geq \Pi^{-1} \quad \text{for all $x\in \Omega$,}\\
		\gamma(x) &\equiv 1 \qquad \, \text{for $\mathrm{dist}(x,\partial \Omega)< \rho$,}
	\end{aligned}
\end{equation}
where we assume knowledge of $\Pi$ and $\rho$. We continuously extend $\gamma\equiv 1$ outside $\Omega$. The data space $Y\subset \mathcal{L}(H^{1/2}(\partial \Omega), H^{-1/2}(\partial \Omega))$ consists of bounded linear operators $\Lambda:H^{1/2}(\partial\Omega)\rightarrow H^{-1/2}(\partial\Omega)$ that are Dirichlet-to-Neumann alike in the sense
\begin{equation}
	\begin{aligned}
		\Lambda(1)&=0,\\
		\int_{\partial \Omega} (\Lambda f)(x) \, d\sigma(x) &= 0 \quad \text{for every $f \in H^{1/2}(\partial \Omega)$.}
	\end{aligned}
\end{equation}
We equip $D(F)$ and $Y$ with the inherited norms $\|\cdot\|_{D(F)} = \|\cdot \|_{L^\infty(\Omega)}$ and $\|\cdot\|_Y = \|\cdot\|_{H^{1/2}(\partial \Omega)\rightarrow H^{-1/2}(\partial \Omega)}$.
\end{assumption}
There is no reason to believe that the regularity assumptions of $\gamma$ is optimal, in fact, we expect that the strategy generalizes to the less regular setting of \cite{carorogers2016}.
We recall the adaptation of the definitions in \cite{engl1996a,kirsch1996a} presented in \cite{Knudsen2009a} of a regularization strategy in the nonlinear setting.
%%%% Definitions %%%%%
\label{def:reg1}
	A family of continuous mappings $\mathcal{R}_\alpha:Y\rightarrow L^\infty(\Omega)$, parametrized by \textit{regularization parameter} $0<\alpha<\infty$, is called a \textit{regularization strategy} for $F$ if
	\begin{equation}\label{reqweak}
		\lim_{\alpha \rightarrow 0} \|\mathcal{R}_\alpha \Lambda_\gamma -  \gamma\|_{L^\infty(\Omega)}=0,
	\end{equation}
	for each fixed $\gamma \in D(F)$. We define the perturbed Dirichlet-to-Neumann map as
	\begin{equation}\label{eq:perturbdata}
		\Lambda_\gamma^\varepsilon = \Lambda_\gamma+\mathcal{E},
	\end{equation}
	with $\mathcal{E}\in Y$ and $\|\mathcal{E}\|_Y \leq \varepsilon$ for some $\varepsilon>0$. We call $\varepsilon$ the noise level, since we eventually simulate perturbations $\mathcal{E}$ as random noise.
\label{def:reg2}
Furthermore, a regularization strategy $\mathcal{R}_\alpha: Y\rightarrow L^\infty(\Omega)$, $0<\alpha<\infty$, is called \textit{admissible} if
\begin{equation}\label{alphaprop}
	\alpha(\varepsilon)\rightarrow 0 \quad \text{ as } \quad \varepsilon \rightarrow 0,
\end{equation}
and for any fixed $\gamma \in \mathcal{D}(F)$ we have
\begin{equation}\label{reqstrong}
\sup_{\Lambda_\gamma^\varepsilon\in Y} \{\|\mathcal{R}_{\alpha(\varepsilon)}\Lambda_\gamma^\varepsilon -\gamma\|_{L^\infty(\Omega)}\mid\|\Lambda_\gamma^\varepsilon-\Lambda_\gamma\|_Y\leq \varepsilon\}\rightarrow 0\quad\text{ as }\quad \varepsilon \rightarrow 0.
\end{equation}

The topology in which we require convergence is essential; we require convergence in strong operator topology, but not in norm topology. The main result of this paper is then as follows.
%%%% Main theorem %%%%%

\begin{theorem}\label{maintheorem}
Suppose $\Pi>0$ and $0<\rho<1$ are given and let $D(F)$ be as in Assumption \ref{assumption2}. Then there exists $\varepsilon_0>0$, dependent only on $\Pi$ and $\rho$ such that the family $\mathcal{R}_\alpha$ defined by \eqref{def:regstrat} is an admissible regularization strategy for $F$ with the following choice of regularization parameter:
\begin{equation}\label{alphadef}
	\alpha(\varepsilon)=\begin{cases}
		(-1/11\log(\varepsilon))^{-1/p} & \text{ for } 0<\varepsilon < \varepsilon_0,\\
		\frac{\varepsilon}{\varepsilon_0}(-1/11\log({\varepsilon_0}))^{-1/p} & \text{ for } \varepsilon \geq \varepsilon_0,
	\end{cases}
\end{equation}
	with $p>3/2$.
\end{theorem}
 This gives theoretical justification for practical reconstruction of the Calderón problem in three dimensions. This is the first deterministic regularization analysis for the three-dimensional Calderón problem known to the authors. Similar results have been shown for the related two-dimensional D-bar reconstruction \cite{Knudsen2009a}, and we will in fact adopt the spectral truncation from there to our setting. {This extension is non-trivial in part because there are no existence and uniqueness guarantees for the CGO solutions that are independent of the magnitude of the complex frequency in the three-dimensional case. In addition, while the two-dimensional D-bar method enjoys the continuous dependence of the solution to the D-bar equation on the scattering transform, it is not obvious when the frequency information of $\gamma$ is stably recovered from the scattering transform corresponding to a perturbed Dirichlet-to-Neumann map in the three-dimensional case. }

We denote the set of bounded linear operators between Banach spaces $X$ and $Y$ by $\mathcal{L}(X,Y)$ and use $\mathcal{L}(X):=\mathcal{L}(X,X)$. We denote the Euclidean matrix operator norm by $\|\cdot\|_N := \|\cdot\|_{\mathbb{C}^{(N+1)^2}\rightarrow \mathbb{C}^{(N+1)^2}}$. The operator norm of $A:H^{s}(\partial \Omega)\rightarrow H^t(\partial \Omega)$ is denoted by $\|A\|_{s,t}$. We reserve $C$ for generic constants and $C_1,C_2,\hdots$ for constants of specific value. Finally, exponential functions of the form $e^{ix\cdot\zeta}$, $x\in \mathbb{R}^3$, $\zeta\in \mathbb{C}^3$, is denoted $e_\zeta(x)$.

In Section \ref{sec:2}, the full non-linear reconstruction algorithm for the three-dimensional Calderón problem is given. Section \ref{sec:3} gives technical estimates regarding the boundary integral equation and the scattering transform and provides a regularizing method for perturbed data with $\varepsilon$ sufficiently small. Then Section \ref{sec:35} extends continuously the method to a regularization strategy $\mathcal{R}_\alpha$ defined on $Y$ and proves Theorem \ref{maintheorem}. In Section \ref{sec:4}, the necessary numerical details concerning the representation of the Dirichlet-to-Neumann map and computation of the relevant norm are given. In addition, a noise model is given. Section \ref{sec:5} presents and discusses numerical results of noise tests with a piecewise constant conductivity distribution using an implementation given in \cite{delbary2014a}, which is available from the corresponding author by request.
\section{The full non-linear reconstruction method}\label{sec:2}
Let $v=\gamma^{1/2}u$, then $v$ is a solution to the Schrödinger equation
\begin{equation}\label{eq:schrod}
	\begin{aligned}
		(-\Delta+q) v &=0 \quad \text { in } \Omega, \\
		v &=g \quad \text { on } \partial \Omega,
	\end{aligned}
\end{equation}
with $q=\gamma^{-1/2}\Delta\gamma^{1/2}$ if and only if $u$ is a solution to \eqref{eq:cond} with $f=\gamma^{-1/2}g$. Note in our setting $q=0$ near $\partial \Omega$ and $q\equiv 0$ is extended continuously outside $\Omega$ and further $\Lambda_q g = \partial_\nu v = \Lambda_\gamma f$. The reconstruction method considered here is based on CGO solutions $\psi_\zeta$ to \eqref{eq:schrod}, which take the form
\begin{equation}\label{cgo1}
	(-\Delta+q)\psi_\zeta = 0 \quad \text{ in } \mathbb{R}^3,
\end{equation}
satisfying $\psi_\zeta(x) = e^{ix\cdot \zeta}(1+r_\zeta(x))$. Here the complex frequency $\zeta \in \mathbb{C}^3$ satisfies $\zeta \cdot \zeta = 0$ making $e^{ix\cdot \zeta}$ harmonic, and the remainder $r_\zeta$ belongs to certain weighted $L^2$ spaces.
In the three-dimensional case, existence and uniqueness of CGO solutions have been shown for large complex frequencies,
\begin{equation}\label{largezeta}
	 |\zeta|>C_0\|q\|_{L^\infty(\Omega)}=:D_q
\end{equation}
for some constant $C_0>0$, or alternatively for $|\zeta|$ small \cite{sylvester1987a,cornean2006a}. The analysis involves the Faddeev Green's function
\begin{equation}
	G_\zeta(x) := e^{i\zeta\cdot x}g_\zeta(x)\qquad g_\zeta(x):=\frac{1}{(2\pi)^3}\int_{\mathbb{R}^3}\frac{e^{ix\cdot\xi}}{|\xi|^2+2\xi\cdot \zeta}\,d\xi,
\end{equation}
where $g_\zeta$ is defined in the sense of the inverse Fourier transform of a tempered distribution and interpretable as a fundamental solution of $(-\Delta-2i\zeta\cdot \nabla)$.
{Boundedness of convolution by $g_\zeta$ on $\Omega$ is well known  \cite{sylvester1987a,brown1996a,salo2006a}:
}
\begin{equation}\label{convgzeta}
	\|g_\zeta * f \|_{L^2(\Omega)}\leq C|\zeta|^{s-1}\|f\|_{L^2(\Omega)}, \quad s\in \{0,1,2\},
\end{equation}
where $|\zeta|$ is bounded away from zero, and $C$ is independent of $\zeta$ and $f$.

The non-physical scattering transform is defined for all those $\zeta$ that give rise to a unique CGO solution $\psi_\zeta$ as
\begin{equation}\label{eq:tformeq}
\mathbf{t}(\xi, \zeta)=\int_{\mathbb{R}^{3}} e^{-i x \cdot(\xi+\zeta)} \psi_{\zeta}(x) q(x)\, d x, \quad \xi \in \mathbb{R}^3.
\end{equation}
It is useful to see the scattering transform as a non-linear Fourier transform of the potential $q$. Indeed, for $|\zeta|>D_q$ we have
\begin{equation}\label{lemmaont}
		|\mathbf{t}(\xi,\zeta)-\hat q(\xi)| \leq C\|q\|_{L^{{\infty}}(\Omega)}^2|\zeta|^{-1},
\end{equation}
for all $\xi \in \mathbb{R}^3$, where $C$ is independent of $\zeta$ and $q$. Whenever $(\zeta+\xi)\cdot (\zeta+\xi)=0$, integration by parts in \eqref{eq:tformeq} yields
\begin{equation}\label{scatter2}
	\mathbf{t}(\xi, \zeta)=\int_{\partial \Omega} e^{-i x \cdot(\xi+\zeta)}(\Lambda_{\gamma}-\Lambda_{1})\psi_\zeta(x)\, d\sigma(x),
\end{equation}
where $d\sigma$ denotes the surface measure on $\partial \Omega$. For fixed $\xi \in \mathbb{R}^3$ this gives rise to the set $\mathcal{V}_\xi = \{\zeta\in \mathbb{C}^3\setminus \{0\} \mid \zeta\cdot \zeta = 0,\, (\zeta+\xi)\cdot (\zeta+\xi) = 0\}$ parametrized by
\begin{equation}\label{zetaxieq}
\zeta(\xi)=\left(-\frac{\xi}{2}+\left(\kappa^{2}-\frac{|\xi|^{2}}{4}\right)^{1 / 2} k^{\perp {\perp}}\right)+i \kappa k^{{\perp}},
\end{equation}
with $\kappa\geq \frac{|\xi|}{2}$ and $k^\perp, k^{\perp\perp} \in \mathbb{R}^3$ are unit vectors and $\{\xi, k^\perp, k^{\perp\perp} \}$ is an orthogonal set \cite{delbary2014a}. Note that for $\zeta(\xi)\in \mathcal{V}_\xi$ and $k\geq \frac{|\xi|}{2}$ we have $|\zeta(\xi)|= \sqrt{2}\kappa$; consequently $\lim_{\kappa\rightarrow \infty} |\zeta(\xi)|=\infty$.

For each fixed $\zeta$ the trace of the CGO solution $\psi_\zeta|_{\partial \Omega}$ is recoverable from the boundary integral equation
\begin{equation}\label{bie}
	\psi_{\zeta}|_{\partial \Omega}+\mathcal{S}_{\zeta}\left(\Lambda_{\gamma}-\Lambda_{1}\right) (\psi_{\zeta}|_{\partial \Omega})=e_{\zeta}|_{\partial \Omega},
\end{equation}
where $\mathcal{S}_{\zeta}: H^{-1/2}(\partial \Omega)\rightarrow  H^{1/2}(\partial \Omega)$ is {the boundary single layer operator} defined by
\begin{equation}\label{fsinglelayer2}
\left(\mathcal{S}_{\zeta} \varphi\right)(x)=\int_{\partial \Omega} G_{\zeta}(x-y) \varphi(y) d \sigma(y),\quad x \in \partial\Omega.
\end{equation}
{With $\mathcal{S}_0$ we denote the boundary single layer operator corresponding to the usual Green’s function $G_0$ for the Laplacian. Occasionally we use the same notation when $x \in \mathbb{R}^3\setminus \partial \Omega$ and note it is well known that $\mathcal{S}_0 \varphi$ and hence $\mathcal{S}_\zeta \varphi$ is continuous in $\mathbb{R}^3$ \cite{colton1992a}.} We let
$$B_\zeta:=[I+\mathcal{S}_{\zeta}\left(\Lambda_{\gamma}-\Lambda_{1}\right)],$$
denote the boundary integral operator and we note the boundary integral equation \eqref{bie} is a uniquely solvable Fredholm equation of the second kind for $|\zeta|>D_q$ \cite{nachman1996a}. This gives a method of recovering the Fourier transform of $q$ in every frequency through the scattering transform \eqref{scatter2} as $|\zeta|\rightarrow \infty$. This method of reconstruction for the Calderón problem in three dimensions was first explicitly given in \cite{nachman1988a,Novikov1988263}. We summarize the method in three steps.
\begin{method}\label{method:1} CGO reconstruction in three dimensions
	\begin{description}
		\item[\textbf{Step} $\mathbf{1}$] Fix $\xi\in \mathbb{R}^3$ and solve the boundary integral equation \eqref{bie} for all $\zeta(\xi)\in \mathcal{V}_\xi$. Compute $\mathbf{t}(\xi,\zeta(\xi))$ by \eqref{scatter2}.
		\item[\textbf{Step} $\mathbf{2}$] Compute $\hat{q}(\xi)$ by
		\begin{equation}
			\lim_{|\zeta(\xi)| \rightarrow \infty} \mathbf{t}(\xi,\zeta(\xi)) = \hat q(\xi), \quad \xi\in\mathbb{R}^3,
		\end{equation}
		and $q(x)$ by the inverse Fourier transform.
		\item[\textbf{Step} $\mathbf{3}$] Solve the boundary value problem
 			\begin{equation}
				\begin{aligned}
					(-\Delta+q) \gamma^{1 / 2} &=0 \quad \text { in } \Omega, \\
					\gamma^{1 / 2} &=1 \quad \text { on } \partial \Omega,
				\end{aligned}
			\end{equation}
			and extract $\gamma$.
	\end{description}
\end{method}
We remark that it is sufficient to solve the boundary integral equation in step 1 for a sequence $\{\zeta_k(\xi)\}_{k=1}^\infty$ of complex frequencies in $\mathcal{V}_\xi$ that tends to infinity.
\section{Regularized reconstruction by truncation}\label{sec:3} We continue by mimicking\break Method \ref{method:1} with $\Lambda_\gamma$ replaced by $\Lambda_\gamma^\varepsilon$ with $\varepsilon$ small. We note that, in any case, using $\psi_\zeta$ with $|\zeta|$ large is impractical. {Indeed, when using perturbed measurements naively in \eqref{scatter2}, the propagated perturbation of $\mathbf{t}$ is $\varepsilon$ multiplied with a factor exponentially growing in $|\zeta|$. This factor originates from the solution of the perturbed boundary integral equation}
\begin{equation}\label{bienoisy}
	B_\zeta^\varepsilon(\psi_{\zeta}^\varepsilon|_{\partial \Omega}):=\psi_{\zeta}^\varepsilon|_{\partial \Omega}+\mathcal{S}_{\zeta}\left(\Lambda^\varepsilon_{\gamma}-\Lambda_{1}\right) (\psi^\varepsilon_{\zeta}|_{\partial \Omega})=e_{\zeta}|_{\partial \Omega},
\end{equation}
{and in multiplication with} $e^{-ix\cdot(\xi+\zeta{(\xi)})}${, see Lemma \ref{lemma3}}. We will show below that \eqref{bienoisy} is solvable for sufficiently small $\varepsilon$. To mitigate this exponential behavior we propose a reconstruction method that makes use of two coupled truncations: one of the complex frequency $\zeta$ and one of the real frequency of the signal $q^\varepsilon$, the perturbed analog of $q$.
  As we shall see, an upper bound of the magnitude $|\zeta(\xi)|$ determines an upper bound of the proximity of $\mathbf{t}$ to $\hat{q}$, when using perturbed data. From \eqref{zetaxieq} we have
\begin{equation}
	|\zeta(\xi)|\geq\frac{|\xi|}{\sqrt{2}},
\end{equation}
and hence fixing $|\zeta(\xi)|$ gives a bounded region in $\mathbb{R}^3$, $|\xi|<M$ for some $M>0$, in which $\mathbf{t}$ can be computed. This gives the following method.

  \begin{method}\label{method:2} Truncated CGO reconstruction in three dimensions
	\begin{description}
		\item[\textbf{Step} $\mathbf{1}^\varepsilon$] Let $M=M(\varepsilon)>0$ be determined by a sufficiently small $\varepsilon$. For each fixed $\xi$ with $|\xi|<M$, take $\zeta(\xi)\in \mathcal{V}_\xi$ with an appropriate size determined by $M$ and solve \eqref{bienoisy} to recover $\psi_{\zeta}^\varepsilon|_{\partial \Omega}$. Compute the truncated scattering transform by
\begin{equation}
	\mathbf{t}^{\varepsilon}_{M(\varepsilon)}(\xi, \zeta(\xi)):= \begin{cases}
	\int_{\partial \Omega} e^{-i x \cdot(\xi+\zeta(\xi))}(\Lambda_{\gamma}^\varepsilon-\Lambda_{1})\psi_\zeta^\varepsilon(x) d\sigma(x), & |\xi|< M(\varepsilon),\\
	0, &|\xi|\geq M(\varepsilon),
	\end{cases}
\end{equation}
		\item[\textbf{Step} $\mathbf{2}^\varepsilon$] Set $\widehat{q^\varepsilon}(\xi):=\mathbf{t}^{\varepsilon}_{M(\varepsilon)}(\xi, \zeta(\xi))$ and compute the inverse Fourier transform to obtain $q^\varepsilon$.
		\item[\textbf{Step} $\mathbf{3}^\varepsilon$] Solve the boundary value problem
\begin{equation}
\begin{aligned}
	(-\Delta+q^\varepsilon)(\gamma^\varepsilon)^{1/2} &=0 \quad && \text { in } \Omega, \\
(\gamma^\varepsilon)^{1/2}&=1 \quad && \text { on } \partial \Omega.
\end{aligned}	
\end{equation}
			and extract $\gamma^\varepsilon$.
	\end{description}
\end{method}
We call $M$ the truncation radius and note it should depend on $\varepsilon$.
Truncation of the scattering transform with truncation radius $M$ is well known in regularization theory for the two-dimensional D-bar reconstruction method \cite{Knudsen2009a}. We can see the real truncation as a low-pass filtering in the frequency domain; this leads to additional smoothing in the spatial domain. Note that $M$ determines the level of regularization and poses as a regularization parameter $\alpha=M^{-1}$ in the sense of \eqref{reqstrong}.

{
 In the following section we derive the required properties of $\mathcal{S}_\zeta$, $B_\zeta^{-1}$ and $(B_\zeta^\varepsilon)^{-1}$. The invertibility of $B_\zeta^\varepsilon$ depends on the invertibility of the unperturbed boundary integral operator $B_\zeta$, which is well known due to the mapping properties of $\mathcal{S}_\zeta$. Although boundedness of $\mathcal{S}_\zeta$ and $B_\zeta^{-1}$ in the three-dimensional case follows by similar arguments to that of the two-dimensional \cite{Knudsen2009a}, it is not immediately clear when $(B_\zeta^\varepsilon)^{-1}$ exists in the absence of existence and uniqueness guarantees of $\psi_\zeta$ for small $|\zeta|$. Neither is it clear under which circumstances $q^\varepsilon$ approximates $q$ as the noise level goes to zero. This is dealt with in Lemma \ref{lemma4} below by choosing a suitable rate, at which $|\zeta|$ and $M$ goes to infinity as $\varepsilon$ goes to zero.
}
\subsection{The perturbed boundary integral equation}\label{sec:per}
{
When $|\zeta|$ is bounded away from zero we can bound $\mathcal{S}_\zeta$ using the mapping properties \eqref{convgzeta} of convolution with $g_\zeta$ between Sobolev spaces defined on $\Omega$. We note that one can give better bounds for arbitrarily small $|\zeta|<1$ than the following result by considering the integral operator $\mathcal{S}_\zeta-\mathcal{S}_0$ with a smooth kernel, see \cite{cornean2006a,Knudsen2009a}.
}
\begin{lemma}\label{lemma1}
	Let $\varphi\in H^{-1/2}\left(\partial \Omega\right)$ such that $\int_{\partial \Omega} \varphi(x)\,d\sigma(x) = 0$ and let $\zeta\in \mathbb{C}^3$ with $\zeta\cdot\zeta=0$ {and $|\zeta|>\beta>0$}. Then for the boundary single layer operator, $\mathcal{S}_\zeta$, we have that
	\begin{equation}\label{upperboundszeta}
		\|\mathcal{S}_\zeta \varphi\|_{H^{1/2}\left(\partial \Omega\right)} \leq C_1 (1+|\zeta|)e^{2|\zeta|}\|\varphi\|_{H^{-1/2}\left(\partial \Omega\right)},
	\end{equation}
	where the constant $C_1$ is independent of $\zeta$.
\end{lemma}
\begin{proof}
We follow \cite{Knudsen2009a}. Letting $x\in \mathbb{R}^3\setminus \overline{\Omega}$ and introducing $u\in H^1(\Omega)$ with $\Delta u = 0$ and $\partial_\nu u = \varphi$ we write

\ \vspace*{-10pt}
\begin{align}
	(\mathcal{S}_\zeta\varphi)(x) &= \int_{\partial \Omega} G_\zeta(x-y)\varphi(y) \, d\sigma(y),\\
	&= \int_{\Omega} \nabla_{y} G_{\zeta}(x-y)\cdot \nabla u(y) d y,\\
&=-\nabla\cdot\left(G_{\zeta} *(\nabla u)\right)(x),\\
&= -\nabla\cdot \left[e^{ix\cdot \zeta}\left(g_{\zeta} *(e^{-iy\cdot \zeta}\nabla u)\right)\right](x),
\end{align}
using integration by parts, the chain rule and the fact that $G_\zeta(x-\cdot)$ is smooth in $\Omega$. By the continuity of $\mathcal{S}_\zeta$ the above holds for $x\in \partial \Omega$ as well. Note from \eqref{convgzeta} and Leibniz' rule that
\begin{equation}
	\|\nabla \cdot \left[e^{ix\cdot \zeta}\left(g_{\zeta} *(e^{-iy\cdot \zeta}\nabla u)\right)\right]\|_{L^2(\Omega)}\leq Ce^{2|\zeta|}\|\nabla u\|_{L^2(\Omega)},
\end{equation}
and
\begin{equation}
	\|\partial_{x_i}\nabla \cdot \left[e^{ix\cdot \zeta}\left(g_{\zeta} *(e^{-iy\cdot \zeta}\nabla u)\right)\right]\|_{L^2(\Omega)}\leq C|\zeta|e^{2|\zeta|}\|\nabla u\|_{L^2(\Omega)},
\end{equation}
for $i=1,2,3$. This yields
\begin{align}
	\|\mathcal{S}_\zeta\varphi\|_{H^{1/2}(\partial\Omega)} &\leq \|\nabla\cdot \left[e^{ix\cdot \zeta}\left(g_{\zeta} *(e^{-iy\cdot \zeta}\nabla u)\right)\right]\|_{H^1(\Omega)},\\
	&\leq C(1+|\zeta|)e^{2|\zeta|}\|\nabla u\|_{L^2(\Omega)},\\
	&\leq C(1+|\zeta|)e^{2|\zeta|}\|\varphi\|_{H^{-1/2}(\partial\Omega)},
\end{align}
using the trace theorem and stability of the Neumann problem for $u$. Here $C$ is dependent on $\beta$ since $|\zeta|>\beta$.
\end{proof}
We have the following estimate of $B_\zeta^{-1}$. The main idea of the proof is to consider a solution $f\in H^{1/2}(\partial \Omega)$ to $B_\zeta f = h$ for some $h\in H^{1/2}(\partial \Omega)$ and then control the exponential component of $f$ by creating a link to the CGO solutions of the Schrödinger equation.
\begin{lemma}\label{lemma2}
For $\zeta\in \mathbb{C}^3\backslash\{0\}$ with $\zeta \cdot \zeta = 0$ and $|\zeta|>D_q$ as in \eqref{largezeta}, the operator $B_\zeta$ is invertible with
\begin{equation}\label{Bzetainv}
	\|B_{\zeta}^{-1}\|_{1 / 2} \leq C_{2} (1+|\zeta|)e^{2|\zeta|},
\end{equation}
where $C_2$ is a constant depending only on the \textit{a priori} knowledge $\Pi$ and $\rho$.
\end{lemma}
\begin{proof}
{	We follow \cite{Knudsen2009a}. Using integration by parts note that $B_\zeta f=f+G_\zeta \ast(qv_f)$ on $\partial \Omega$, where $v_f\in H^1(\Omega)$ is the unique solution to
\begin{equation}
\begin{aligned}
	(-\Delta+q)v_f &=0 \quad && \text { in } \Omega, \\
	v_f &=f \quad && \text { on } \partial \Omega.
\end{aligned}
\end{equation}
To bound $f$ we bound $v_f$ by writing $v_f = v-u^{\mathrm{exp}}$ with
\begin{equation}
\begin{aligned}
	\Delta v &=0 \quad && \text { in } \Omega, \\
	v &= B_\zeta f \quad && \text { on } \partial \Omega,
\end{aligned}
\end{equation}
and $u^{\mathrm{exp}}:=G_\zeta *(qv_f)$. From the stability property of the Dirichlet problem it is sufficient to bound $u^{\mathrm{exp}}$ in terms of $v$. Note $(-\Delta+q)u^{\mathrm{exp}}=qv$ and hence conjugating with exponentials yields the equation in $\mathbb{R}^3$,
\begin{equation}\label{eq:condconj}
	(-\Delta-2i\zeta \cdot \nabla+q)u = qve^{-ix\cdot \zeta},
\end{equation}
where we set $u = e^{-ix\cdot \zeta}u^{\mathrm{exp}}$. It is well known that $u$ is the unique solution among functions in certain weighted $L^2(\mathbb{R}^3)$-spaces satisfying
$$\|u\|_{L^2(\Omega)} \leq C\|q\|_{L^\infty}\frac{e^{|\zeta|}}{|\zeta|}\|v\|_{L^2(\Omega)},$$
whenever $|\zeta|>D_q$, see \cite{sylvester1987a}. Indeed, convolution with $g_\zeta$ on both sides of \eqref{eq:condconj} gives
$$u = g_\zeta*(-qu+qve^{-ix\cdot \zeta}),$$
which upgrades the estimate to
$$\|u\|_{H^1(\Omega)} \leq C\|q\|_{L^\infty}e^{|\zeta|}\|v\|_{L^2(\Omega)},$$
using \eqref{convgzeta}. Now the estimate \eqref{Bzetainv} follows straightforwardly from the trace theorem.
}
\end{proof}
We note that a main difference between the boundary integral equation in two dimensions and three dimensions is the possible existence of a certain $\zeta$ for which there exists no unique CGO solutions to \eqref{cgo1}. The next result shows that Lemma \ref{lemma1} and Lemma \ref{lemma2} implies solvability of the perturbed boundary integral equation using a Neumann series argument on the form
\begin{equation}
B_{\zeta}^\varepsilon = I+\mathcal{S}_\zeta(\Lambda_\gamma^\varepsilon-\Lambda_\gamma) +\mathcal{S}_\zeta(\Lambda_\gamma-\Lambda_1)=[I+A_\zeta^\varepsilon] B_\zeta,
\end{equation}
where $A^\varepsilon_\zeta:=\mathcal{S}_\zeta \mathcal{E} B_\zeta^{-1}$ is a bounded operator in $H^{1/2}(\partial \Omega)$. {It is clear from Lemma \ref{lemma2} that $q$ fixes a lower bound for $|\zeta|$, for which $B_\zeta$ is certain to be invertible.} When the noise level is sufficiently small such that $D_q<|\zeta|<R(\varepsilon)$, for some $R$, we may invert $B_\zeta^\varepsilon$. We have the following result.
\begin{lemma}\label{lemma3}
	Let $R = R(\varepsilon):=-\frac{1}{6}\log{\varepsilon}$, and suppose $D_q<|\zeta |< R(\varepsilon_1)$ for some $0<\varepsilon_1<1$. Then there exists $0<\varepsilon_2\leq\varepsilon_1$ for which $B_\zeta^\varepsilon$ is invertible whenever $0<\varepsilon < \varepsilon_2$. Furthermore we have the estimate
	\begin{equation}
		\|\psi^\varepsilon_\zeta - \psi_\zeta\|_{H^{1/2}(\partial \Omega)} \leq C_3\varepsilon(1+R)^4e^{7R},
	\end{equation}
	where $C_3$ is a constant depending only on the \textit{a priori} knowledge of $\Pi$ and $\rho$.
\end{lemma}
\begin{proof}
Since $\mathcal{E}\in Y$, it maps onto trace functions that have zero mean on the boundary. Then from Lemma \ref{lemma1} and Lemma \ref{lemma2} we find
\begin{align}\label{epsilon0}
	\|A_\zeta^\varepsilon\|_{1/2}= \|\mathcal{S}_\zeta \mathcal{E} B_\zeta^{-1}\|_{1/2}&\leq C_1C_2 \varepsilon (1+R)^2 e^{4R},\\
	&\leq C \varepsilon e^{5R}, \label{eq:rhs}
\end{align}
where we have absorbed the polynomial in $R$ into the exponential and thereby obtained a new constant. By the definition of $R$, we note the right-hand side of \eqref{eq:rhs} goes to zero as $\varepsilon$ goes to zero, and hence there exists a $0<\varepsilon_2 \leq \varepsilon_1$ such that $\|A_\zeta^\varepsilon\|_{1/2}<\frac{1}{2}$. Then by a Neumann series argument, $I+A_\zeta^\varepsilon$ is invertible with $\|(I+A_\zeta^\varepsilon)^{-1} \|_{1/2}<2$, and $(B_\zeta^\varepsilon)^{-1}=B_\zeta^{-1}[I+A_\zeta^\varepsilon]^{-1}$. From the boundary integral equations we have $\psi_\zeta=B^{-1}_\zeta (e_\zeta|_{\partial \Omega})$ and $\psi^\varepsilon_\zeta=(B_\zeta^\varepsilon)^{-1}(e_\zeta|_{\partial \Omega})$. Then with the use of Lemma \ref{lemma2} we have for $0<\varepsilon<\varepsilon_2$
	\begin{align}
		\|\psi^\varepsilon_\zeta\|_{H^{1/2}(\partial \Omega)} &\leq \|(B_\zeta^\varepsilon)^{-1}(e_\zeta|_{\partial \Omega})\|_{H^{1/2}(\partial \Omega)},\\		
		&\leq 2\|B_\zeta^{-1}\|_{1/2} \|e^{ix\cdot \zeta}\|_{H^{1/2}(\partial \Omega)},\label{Bzetaeps}\\
		&\leq C (1+|\zeta|)^2e^{3|\zeta|}.\label{psiestimate}
	\end{align}
With the use of Lemma \ref{lemma2} we have for $0<\varepsilon<\varepsilon_2$
\begin{align}
	\|(B_\zeta^\varepsilon)^{-1}-B^{-1}_\zeta\|_{1/2}&=\|B_\zeta^{-1}[(I+A_\zeta^\varepsilon)^{-1}-I]\|_{1/2},\\
	&\leq \|B_\zeta^{-1}\|_{1/2}\|(I+A_\zeta^\varepsilon)^{-1}[I-(I+A_\zeta^\varepsilon)]\|_{1/2},\\
	&\leq \|B_\zeta^{-1}\|_{1/2}\|(I+A_\zeta^\varepsilon)^{-1}\|_{1/2}\|A_\zeta^\varepsilon\|_{1/2},\\
	&\leq 2 C_1C_{2}^2 \varepsilon(1+R)^3e^{6R}.
\end{align}
Finally we obtain
	\begin{align}
		\|\psi^\varepsilon_\zeta - \psi_\zeta\|_{H^{1/2}(\partial \Omega)} &=\|[(B_\zeta^\varepsilon)^{-1}-B^{-1}_\zeta]e_\zeta\|_{H^{1/2}(\partial \Omega)},\\
		&\leq \|(B_\zeta^\varepsilon)^{-1}-B^{-1}_\zeta\|_{1/2} \|e^{ix\cdot \zeta}\|_{H^{1/2}(\partial \Omega)},\\
		&\leq 2 C_1C_{2}^2 \varepsilon(1+R)^4e^{7R},\label{psidifestimate}
	\end{align}
for $0<\varepsilon<\varepsilon_2$.
\end{proof}

\subsection{Truncation of the scattering transform}\label{sec:trunc}
We now show that fixing the magnitude of the complex frequency $|\zeta(\xi)|=(M(\varepsilon))^p$ with $p>3/2$, enables control over the proximity of the truncated scattering transform $\mathbf{t}^\varepsilon_M(\cdot,\zeta)$ to $\hat q$ for small noise levels. This choice is justified from the following result.

\begin{lemma}\label{lemma4}
	Let $M(\varepsilon) = (-1/11 \log(\varepsilon))^{1/p}$ be a truncation radius depending on $\varepsilon$ and some exponent $p>3/2$. Fix $\xi \in \mathbb{R}^3$ with $|\xi|< M(\varepsilon)$, suppose $\zeta(\xi)\in \mathcal{V}_\xi$ with
\begin{equation}
	|\zeta(\xi)|=(M(\varepsilon))^p=-\frac{1}{11}\log(\varepsilon)
	\end{equation}
and let $\varepsilon_2$ be defined as in the proof of Lemma \ref{lemma3}. Further fix $q\in L^\infty(\Omega)$ corresponding to a $\gamma \in D(F)$. Then $\mathbf{t}^\varepsilon_M$ is well defined by \eqref{errtscat} for $0<\varepsilon<\varepsilon_2$ and
\begin{equation}
	\lim_{\varepsilon \rightarrow 0} \|\mathbf{t}^\varepsilon_{M(\varepsilon)}-\hat q\|_{L^2(\mathbb{R}^3)}=0.
\end{equation}
\end{lemma}
\begin{proof}
For \textit{(i)} fix first $|\xi|< M(\varepsilon)$ and note first by the triangle inequality that
\begin{equation}\label{trianglerelation}
	|\mathbf{t}^{\varepsilon}_{M(\varepsilon)}(\xi,\zeta(\xi))-\hat q(\xi)|\leq |\mathbf{t}^\varepsilon_{M(\varepsilon)}(\xi,\zeta(\xi))-\mathbf{t}(\xi,\zeta(\xi)) |+ |\mathbf{t}(\xi,\zeta(\xi))-\hat q(\xi)|.
\end{equation}
By Lemma \ref{lemma3} there exists a unique solution $\psi^\varepsilon_\zeta$ to the perturbed boundary integral equation and hence $\mathbf{t}^\varepsilon_M$ is well defined. Using \eqref{psiestimate} and \eqref{psidifestimate}, we find the following, in which we set $R=R(\varepsilon)$, $M=M(\varepsilon)$ and $\zeta = \zeta(\xi)$ for simplicity of exposition,
\begin{align}
	|\mathbf{t}^\varepsilon_M(\xi,\zeta)-\mathbf{t}(\xi,\zeta)| &= \left|\int_{\partial \Omega} e^{-ix\cdot(\xi+\zeta)}[(\Lambda^\varepsilon_\gamma-\Lambda_1)\psi_\zeta^\varepsilon(x)-(\Lambda_\gamma-\Lambda_1)\psi_\zeta(x)]d\sigma(x) \right|,\\
	&\leq \|e ^{-ix\cdot(\xi+\zeta)}\|_{H^{1/2}(\partial \Omega)}\|\Lambda_\gamma-\Lambda_1\|_Y\|\psi_\zeta^\varepsilon-\psi_\zeta\|_{H^{1/2}(\partial \Omega)}\\ \label{testimation}
	&\phantom{=}\,\,+ \|e ^{-ix\cdot(\xi+\zeta)}\|_{H^{1/2}(\partial \Omega)}\|\Lambda_\gamma^\varepsilon-\Lambda_\gamma\|_Y\|\psi_\zeta^\varepsilon\|_{H^{1/2}(\partial \Omega)},\\
	&\leq C(1+|\zeta|)e^{|\zeta|}\left[\varepsilon(1+|\zeta|)^4e^{7|\zeta|} + \varepsilon(1+|\zeta|)^2e^{3|\zeta|} \right ],
\end{align}
where we use the fact that $\|\Lambda_\gamma - \Lambda_1\|_Y \leq C$, where $C$ depends only on $\Pi$ by the continuity of the forward map $\gamma\mapsto \Lambda_\gamma$. Then,
\begin{equation}
		|\mathbf{t}^\varepsilon_M(\xi,\zeta)-\mathbf{t}(\xi,\zeta)|\leq C\varepsilon e^{9|\zeta|}.
\end{equation}
 Using \eqref{trianglerelation} and the property \eqref{lemmaont} we conclude for $|\xi|< M(\varepsilon)$ that
\begin{equation}\label{eq:iestimate}
	|\mathbf{t}^{\varepsilon}_M(\xi,\zeta)-\hat q(\xi)|\leq C\varepsilon e^{9|\zeta|}+ C|\zeta|^{-1}.
	\end{equation}
Then for \textit{(ii)}, using the triangle inequality and \eqref{eq:iestimate} we find
\begin{align}
	\|\mathbf{t}^\varepsilon_M-\hat q\|_{L^2(\mathbb{R}^3)}&\leq \|\mathbf{t}^\varepsilon_M-\hat q\|_{L^2(|\xi|< M)}+\|\hat q\|_{L^2(|\xi|\geq M)},\\
	&\leq C(\varepsilon e^{9|\zeta|}+M^{-p})\left(\int_{0}^M  r^2 \, dr\right)^{1/2}+\|\hat q\|_{L^2(|\xi|\geq M)},\\
	& \leq C(\varepsilon e^{10|\zeta|}+M^{3/2-p})+\|\hat q\|_{L^2(|\xi|\geq M)},\\
	&\leq C\varepsilon^{1/11}+C(-1/11\log(\varepsilon))^{3/2-p}+\|\hat q\|_{L^2(|\xi|\geq M)},
\end{align}
for $0<\varepsilon<\varepsilon_2$. Since $q\in L^\infty(\Omega)$ is compactly supported in $\Omega$, we have $q\in L^2(\mathbb{R}^3)$, and hence the energy of the tail of $\hat q$ converges to zero as $M(\varepsilon)$ goes to infinity. The result follows as $p>3/2$.
\end{proof}
One may obtain an explicit decay of $\hat q$ by assuming a certain regularity of $q$. Notice the proof above works fine with the choice $|\zeta| = K_1M^{p}+K_2$ for some $0<K_1<1$, $K_2>0$ and $p>3/2$. A user may choose among such $|\zeta|$ freely, with $p=3/2$ being the critical choice. We now prove that $\gamma^\varepsilon$ exists and is unique and that the propagated reconstruction error tends to zero if $\varepsilon\rightarrow 0$, given $\|q^\varepsilon-q\|_{L^2(\Omega)}$ is sufficiently small. This is possible in $H^2(\Omega)$ by a Neumann series argument and elliptic regularity. For the boundary value problem
\begin{equation}
\begin{aligned}
	(-\Delta+q^\varepsilon)u&=f \quad && \text { in } \Omega, \\
u&=0 \quad && \text { on } \partial \Omega,
\end{aligned}	
\end{equation}
with $f\in L^2(\Omega)$, we introduce the notation $L^\varepsilon: H^1_0(\Omega)\cap H^2(\Omega)\rightarrow L^2(\Omega)$, $L^\varepsilon: u\mapsto f$, defined for any $q^\varepsilon\in L^2(\Omega)$ and then note
\begin{equation}\label{eq:finalstep}
	\gamma^\varepsilon = [(L^\varepsilon)^{-1}(-q^\varepsilon)+1]^2,
\end{equation}
whenever $(L^\varepsilon)^{-1}$ exists.

\begin{lemma}\label{lemma5}
Let $q=\Delta \gamma^{1/2}\gamma^{-1/2}$ be a potential with $\gamma\in D(F)$. Then there exists a $0<\varepsilon_3<1$ such that for $0<\varepsilon<\min(\varepsilon_2,\varepsilon_3)=:\varepsilon_0$ the boundary value problem
\begin{equation}\label{bvp1}
\begin{aligned}
	(-\Delta+q^\varepsilon)(\gamma^\varepsilon)^{1/2}&=0 \quad && \text { in } \Omega, \\
	(\gamma^\varepsilon)^{1/2} &=1 \quad && \text { on } \partial \Omega,
\end{aligned}
\end{equation}
has a unique solution in $H^2(\Omega)$. Furthermore the following inequality holds
\begin{equation}\label{estimatelem5}
\|\gamma^{1 / 2}-(\gamma^{\varepsilon})^{1 / 2}\|_{H^{2}(\Omega)} \leq C_4\|q-q^{\varepsilon}\|_{L^{2}(\Omega)},
\end{equation}
where $C_4$ is dependent only on $\Pi$ and $\rho$.
\end{lemma}
\begin{proof}
Note $(-\Delta+q)^{-1}$ exists and is bounded for $L^2(\Omega)$ into $H^1_0(\Omega)\cap H^2(\Omega)$ with
 \begin{equation}\label{boundedsol}
\|u\|_{H^{2}(\Omega)} \leq C\|f\|_{L^{2}(\Omega)},
\end{equation}
by elliptic regularity \cite{evans2010a}. Here $C$ is dependent only on $\Pi$. We construct
	\begin{equation}\label{construction}
		L^\varepsilon u = (-\Delta+q)[I+(-\Delta+q)^{-1}(q^\varepsilon-q)]u,
	\end{equation}
	and seek boundedness of $(-\Delta+q)^{-1}(q^\varepsilon-q)$ in $H^2(\Omega)$ as our goal. For any $u\in H^2(\Omega)$
\begin{align}
\|(-\Delta+q)^{-1}(q^\varepsilon-q)u\|_{H^2(\Omega)} \leq C \|q^\varepsilon-q\|_{L^2(\Omega)} \|u\|_{H^2(\Omega)},
\end{align}
using \eqref{boundedsol} and Sobolev embedding theory. By Lemma \ref{lemma4}, there exists a $0<\varepsilon_3<1$ such that for all $0<\varepsilon<\min(\varepsilon_2,\varepsilon_3)$
\begin{equation}
	\|(-\Delta+q)^{-1}(q^\varepsilon-q)\|_{H^{2}(\Omega)\rightarrow H^{2}(\Omega)}\leq C\|q^\varepsilon-q \|_{L^2(\Omega)}<\frac{1}{2}.
\end{equation}
 Hence $(L^\varepsilon)^{-1}$ exists and is uniformly bounded with respect to $0<\varepsilon \leq \varepsilon_0$. Finally, since $\gamma \in L^\infty(\Omega)$ we have $(q^\varepsilon-q)\gamma^{1/2}\in L^2(\Omega)$, and by solving
\begin{alignat}{2}
	L^\varepsilon(\gamma^{1/2}-(\gamma^\varepsilon)^{1/2})&=(q^\varepsilon-q)\gamma^{1/2} \quad && \text { in } \Omega \\
\gamma^{1/2}-(\gamma^\varepsilon)^{1/2}&=0 \quad && \text { on } \partial \Omega,
\end{alignat}
we obtain the estimate \eqref{estimatelem5}.
\end{proof}
We conclude that $\gamma^\varepsilon$ of Method \ref{method:2} exists uniquely and approximates $\gamma$ in the $H^2(\Omega)$-norm, whenever $\varepsilon < \varepsilon_0$.
\section{Extending the method to a regularization strategy}\label{sec:35}
From the definition of an admissible regularization strategy it is clear $\mathcal{R}_\alpha$ must be defined on $Y$ and not only an $\varepsilon_0$-neighborhood of $F(\mathcal{D}(F))$. However, $(B_\zeta^\varepsilon)^{-1}$ and $(L^\varepsilon)^{-1}$ exists only for small enough $\varepsilon$. We confront this by extending these operators to $(B_\zeta^\varepsilon)_\alpha^{\dagger}$ and $(L^\varepsilon)_\alpha^{\dagger}$ coinciding with $(B_\zeta^\varepsilon)^{-1}$ and $(L^\varepsilon)^{-1}$ for $\varepsilon < \varepsilon_0$, such that $\mathcal{R}_\alpha$ is continuous and well defined on $Y$. There are several ways to obtain such extensions, however we will follow \cite{Knudsen2009a} and construct explicit pseudoinverses by means of functional calculus.
Define the normal operator
\begin{equation}
	S^\varepsilon_\zeta:=(B_\zeta^\varepsilon)^\ast(B_\zeta^\varepsilon)\in\mathcal{L}(H^{1/2}(\partial \Omega)),
\end{equation}
where $(B_\zeta^\varepsilon)^\ast$ is the adjoint operator of  $(B_\zeta^\varepsilon)\in \mathcal{L}(H^{1/2}(\partial \Omega))$. Similarly we define
\begin{equation}
	T^\varepsilon_\zeta:=(L^\varepsilon)^\ast(L^\varepsilon)\in\mathcal{L}(L^2(\Omega)).
\end{equation}
Let $h_\alpha^1$ and $h_\alpha^2$ be two real functions defined for $0<\alpha<\infty$ as
\begin{equation}
	h^i_\alpha(t) := \begin{cases}
		t^{-1} & \text{ for } t>\kappa_i(\alpha),\\
		\kappa_i(\alpha)^{-1} & \text{ for } t\leq \kappa_i(\alpha),
	\end{cases}
\end{equation}
for $i=1,2$ with $\kappa_i(\alpha)=\frac{1}{4}r_i(\alpha)^2$, where we will see below the estimates \eqref{estimatelem5} and \eqref{boundB} motivates the definition
\begin{equation}
	r_i(\alpha) := \begin{cases}
		\frac{1}{C_2(1+\alpha^{-p})e^{2\alpha^{-p}}} & \text{ for } i=1,\\
		\frac{1}{C_4} & \text{ for } i=2,
	\end{cases}
\end{equation}
with $p>3/2$. We define the $\alpha$-pseudoinverses ${(B_\zeta^\varepsilon)}_\alpha^{\dagger}$ of $B_\zeta^\varepsilon$ and ${(L^\varepsilon)}_\alpha^{\dagger}$ of $L^\varepsilon$ for any $0<\alpha<\infty$ as
\begin{align}
	(B_\zeta^\varepsilon)_\alpha^{\dagger} &:= h^1_\alpha(S_\zeta^\varepsilon)(B_\zeta^\varepsilon)^\ast,\\
	(L^\varepsilon)_\alpha^{\dagger} &:= h^2_\alpha(T^\varepsilon)(L^\varepsilon)^\ast,
\end{align}
where the operators $h^1_\alpha(S_\zeta^\varepsilon)$ in $\mathcal{L}(H^{1/2}(\partial \Omega))$ and $h^2_\alpha(T^\varepsilon)$ in $\mathcal{L}(L^2(\Omega))$ are defined in the sense of continuous functional calculus (see for example \cite{reed1980a,so2018a}) and depend continuously on $S_\zeta^\varepsilon$ and $T^\varepsilon$, respectively (see for example \cite[Lemma 3.1]{Knudsen2009a}). This implies $\Lambda_\gamma^\varepsilon \mapsto (B_\zeta^\varepsilon)_\alpha^{\dagger}$ and $q^\varepsilon \mapsto (L^\varepsilon)_\alpha^{\dagger}$ are continuous mappings. Explicitly, for a self-adjoint operator $S:\mathcal{H}\rightarrow \mathcal{H}$ for a Hilbert space $\mathcal{H}$ we set
\begin{equation}\label{spectraldecomp}
	h^i_\alpha(S) = \int_{\sigma(S)} h^i_\alpha(\lambda)\, dP(\lambda),
\end{equation}
where $\sigma(S)\subset \mathbb{C}$ denotes the spectrum of $S$, and $P$ is a spectral measure on $\sigma(S)$.

\begin{method}\label{Method2} Regularized CGO reconstruction in three dimensions
\begin{description}
		\item[\textbf{Step} $\mathbf{1_\alpha}$] Given $\alpha>0$, set $M=\alpha^{-1}$. For each $|\xi|<M$ take $\zeta(\xi)\in \mathcal{V}_\xi$ with $|\zeta(\xi)|=M^{p}$ for $p>3/2$ and define
		\begin{equation}
			\tilde{\psi}_\alpha:=(B_\zeta^\varepsilon)_\alpha^{\dagger}(e_\zeta|_{\partial \Omega})
		\end{equation}
and compute the truncated scattering transform $\mathbf{t}_\alpha(\xi,\zeta(\xi))$ for $\zeta(\xi)$ in $\mathcal{V}_\xi$  by
\begin{equation}\label{errtscat}
\tilde{\mathbf{t}}_\alpha(\xi, \zeta(\xi))= \begin{cases}
	\int_{\partial \Omega} e^{-i x \cdot(\xi+\zeta(\xi))}(\Lambda_{\gamma}^\varepsilon-\Lambda_{1})\tilde{\psi}_\alpha(x) d\sigma(x) & |\xi|< M,\\
	0 &|\xi|\geq M
\end{cases}
\end{equation}
		\item[\textbf{Step} $\mathbf{2_\alpha}$] Define $\widehat{q_\alpha}(\xi):=\tilde{\mathbf{t}}_\alpha(\xi,\zeta(\xi))$ and compute the inverse Fourier transform to obtain $q_\alpha$.
		\item[\textbf{Step} $\mathbf{3_\alpha}$] Solve the boundary value problem \eqref{bvp1} by computing $(L^\varepsilon)_\alpha^{\dagger}(-q_\alpha)$ and set
		\begin{equation}\label{def:regstrat}
			\mathcal{R}_\alpha \Lambda_\gamma^\varepsilon := [(L^\varepsilon)_\alpha^{\dagger}(-q_\alpha)+1]^2
		\end{equation}
\end{description}
\end{method}

% THEOREM 1
% proof
\begin{proof}[Proof of Theorem \ref{maintheorem}]
Given $\Lambda_\gamma^\varepsilon$ in $Y$ we have
\begin{align}
	|\tilde{\mathbf{t}}_\alpha(\xi, \zeta(\xi))| &\leq \left|\int_{\partial \Omega} e^{-ix\cdot(\xi+\zeta)}[(\Lambda^\varepsilon_\gamma-\Lambda_\gamma)\tilde{\psi}_\alpha(x)+(\Lambda_\gamma-\Lambda_1)\tilde{\psi}_\alpha(x)]d\sigma(x) \right|,\\
	&\leq \|e ^{-ix\cdot(\xi+\zeta)}\|_{H^{1/2}(\partial \Omega)}\|\Lambda^\varepsilon_\gamma-\Lambda_\gamma\|_Y\|\tilde{\psi}_\alpha\|_{H^{1/2}(\partial \Omega)}\\
	&\phantom{=}\,\, \|e ^{-ix\cdot(\xi+\zeta)}\|_{H^{1/2}(\partial \Omega)}\|\Lambda_\gamma-\Lambda_1\|_Y\|\tilde{\psi}_\alpha\|_{H^{1/2}(\partial \Omega)},\\
	&< \infty,\label{teruendelig}
\end{align}
for all $\xi \in \mathbb{R}^3$, since $(B_\zeta^\varepsilon)_\alpha^{\dagger}$ is bounded in $H^{1/2}(\partial\Omega)$. Then by compact support $\tilde{\mathbf{t}}_\alpha\in L^2(\mathbb{R}^3)$. It follows the inverse Fourier transform of this object is well defined and hence the family of operators $\mathcal{R}_\alpha$ is well defined. Using the continuity of the maps $\Lambda_\gamma^\varepsilon \mapsto (B_\zeta^\varepsilon)_\alpha^{\dagger}$ and $q_\alpha \mapsto (L^\varepsilon)_\alpha^{\dagger}$, a parallel estimation to \eqref{testimation} and the linearity and boundedness of the inverse Fourier transform in $L^2(\mathbb{R}^3)$, it is clear that $\mathcal{R}_\alpha$ is a family of continuous mappings. Now recall from Lemma \ref{lemma2} and \eqref{Bzetaeps} that for $0<\varepsilon<\varepsilon_0$ we have that
\begin{equation}\label{boundB}
	\|(B_\zeta^\varepsilon)\|_{1/2}^{-1} \leq \|(B_\zeta^\varepsilon)^{-1}\|_{1/2}\leq 2 C_{2} (1+|\zeta|)e^{2|\zeta|}.
\end{equation}
Set $|\zeta|=\alpha^{-p}$ and note
\begin{equation}
	S_\zeta^\varepsilon \geq \frac{1}{4}r_1(\alpha)^2 I.
\end{equation}
By definition of the $\alpha$-pseudoinverse and \eqref{spectraldecomp} we have that $(B_\zeta^\varepsilon)_\alpha^{\dagger}=(B_\zeta^\varepsilon)^{-1}$ for $0<\varepsilon<\varepsilon_0$, and hence $\tilde{\psi}_\alpha=\psi_\zeta^\varepsilon$ is unique. It follows by Lemma \ref{lemma4} that $\tilde{\mathbf{t}}(\cdot,\zeta(\cdot))$ is well defined and $q_\alpha = q^\varepsilon$ converges to $q$ as $\varepsilon$ goes to zero. Conversely, for $0<\varepsilon<\varepsilon_0$ we have $(L^\varepsilon)_\alpha^{\dagger}=(L^\varepsilon)^{-1}$, and hence by Lemma \ref{lemma5} and the Sobolev embedding $H^2(\Omega)\subset C^0(\overline{\Omega})$, \eqref{reqstrong} is satisfied. Note also the weaker requirement \eqref{reqweak} follows analogously. The property \eqref{alphaprop} is satisfied by \eqref{alphadef}.
\end{proof}
A direct consequence of the truncation of the scattering transform is the following property of the reconstruction $\mathcal{R}_\alpha(\varepsilon) \Lambda_\gamma^\varepsilon$ for sufficiently small $\varepsilon$. The regularized reconstructions are as regular as $\Omega$.
\begin{proposition}
	Suppose $\Lambda_\gamma^\varepsilon=\Lambda_\gamma+\mathcal{E}$ with $\|\mathcal{E}\|_Y\leq \varepsilon < \varepsilon_0$. Then $\mathcal{R}_\alpha(\varepsilon) \Lambda_\gamma^\varepsilon\in C^\infty(\overline{\Omega})$.
\end{proposition}
\begin{proof}
	Since $\tilde{\mathbf{t}}_\alpha(\cdot, \zeta(\cdot))\in L^1(\mathbb{R}^3)$ has compact support, it follows $q_\alpha$ is smooth. Since $\partial \Omega$ is smooth, it follows $\mathcal{R}_\alpha \Lambda_\gamma^\varepsilon\in C^\infty(\overline{\Omega})$ by elliptic regularity \cite{evans2010a}.
\end{proof}

\section{Computational methods}\label{sec:4}
In this section we outline methods of representing and computing the Dirichlet-to-Neumann map numerically and consider the discretization of the boundary integral equations. We assume $\Omega = B(0,1)$ in order to utilize spherical harmonics in representation of functions on $\partial \Omega$.
\subsection{Representation and computation of the Dirichlet-to-Neumann map}\label{sec41}
 We consider the Hilbert space $H^s(\partial \Omega)$, $s>0$, defined as the space of all functions $f$ in $L^2(\partial \Omega)$ that satisfy
\begin{equation}\label{h1norm1}
	\|f\|_{L^{2}(\partial \Omega)}^{2}+\|(-\Delta_S)^{s / 2} f\|_{L^{2}(\partial \Omega)}^{2}<\infty,
\end{equation}
where $(-\Delta_S)^{s / 2}$ is the fractional order spherical Laplace operator on the unit sphere. Since spherical harmonics, say $\{Y_n^m\}_{n\in\mathbb{N}_0, |m|\leq n}$, constitute an orthonormal basis of $L^2(\partial \Omega)$ (see for example \cite{colton1992a}), we may expand $f \in L^2(\partial \Omega)$ as
\begin{equation}
	f = \sum_{n=0}^\infty\sum_{m=-n}^n \langle f, Y^m_n \rangle Y^m_n, \qquad \langle f, Y^m_n \rangle = \int_{\partial \Omega} f(x) \overline{Y^m_n(x)}\, d\sigma(x).
\end{equation}
The spherical harmonics are eigenvectors of $(-\Delta_S)$, in particular,
\begin{equation}
\left(-\Delta_{S}\right)^{s / 2} Y=(n(n+1))^{s / 2} Y,
\end{equation}
for any spherical harmonic $Y$ of degree $n$. Then the requirement \eqref{h1norm1} gives rise to a characterization of $H^s(\partial \Omega)$ suitable for $s\in \mathbb{R}$ as those functions $f \in L^2(\partial \Omega)$ that satisfy
\begin{equation}
	\sum_{n=0}^\infty\sum_{m=-n}^n (1+n^2)^s|\langle f, Y^m_n \rangle |^2 < \infty.
\end{equation}
See \cite[Chapter 1.7]{MR0350177} for a more general treatment and the case $s<0$. Thus we define the $H^s(\partial \Omega)$ inner products as
\begin{equation}
	\langle f,g \rangle_s:=\langle f,g \rangle_{H^s(\partial \Omega)}=\sum_{n=0}^\infty\sum_{m=-n}^n w_s(n)\langle f,Y^m_n\rangle \overline{w_s(n) \langle g,Y^m_n\rangle},
\end{equation}
where the multiplier functions are defined as
\begin{equation}
	w_s(n) :=(1+n^2)^{s/2}, \qquad \text{for $n\in \mathbb{N}_0$, $s\in \mathbb{R}$},
\end{equation}
and hence $\|f\|_{H^s(\partial \Omega)} = \langle f, f \rangle_s^{1/2}$. We build an orthonormal basis $\{\phi_{n,m}^s\}_{n\in \mathbb{N}_0, |m|\leq n}$ of $H^s(\partial \Omega)$ with
\begin{equation}
	\phi_{n,m}^s = w_{-s}(n)Y^m_n.
\end{equation}
and hence any $g\in H^s(\partial\Omega)$ has the expansion
\begin{equation}
	g = \sum_{n=0}^\infty\sum_{m=-n}^n \langle g, \phi_{n,m}^s \rangle_s \phi_{n,m}^s.
\end{equation}
Consider the $L^2(\partial \Omega)$ orthogonal projection $P_N$ to the space spanned by spherical harmonics of degree less than or equal to $N$, as
\begin{equation}
	P_Ng = \sum_{n=0}^N\sum_{m=-n}^n \langle g, Y^m_n\rangle Y^m_n.
\end{equation}
Note $\langle g, Y^m_n\rangle$ as an integral over the unit sphere may be approximated by coefficients $c_{n,m}(\underline{g})$ using Gauss-Legendre quadrature in $2(N+1)^2$ appropriately chosen quadrature points $\{x_k\}_{k=1}^{2(N+1)^2}$ on the unit sphere as in \cite{delbary2014a}. Here we denote $\underline{g} = (g(x_1),\hdots, g(x_{2(N+1)^2}))$. Define
\begin{equation}
	L_N g := \sum_{n=0}^N\sum_{m=-n}^n c_{n,m}(\underline{g}) Y^m_n.
\end{equation}
We may approximate any operator $\Lambda:H^{s}(\partial \Omega)\rightarrow H^{-s}(\partial \Omega)$ using $Q$, a matrix in $\mathbb{C}^{2(N+1)^2\times 2(N+1)^2}$ defined by
\begin{equation}\label{lambdaapprox}
(\Lambda g)(x_k) \simeq [{ Q}\underline{g}]_k :=  \sum_{n=0}^N\sum_{m=-n}^n  c_{n,m}(\underline{g}) (\Lambda Y_n^m)(x_k), \quad k=1,\hdots,2(N+1)^2.
\end{equation}
{From here it is clear we can write ${ Q}$ as
\begin{equation}
	{Q} = \widetilde{{Q}}{A}, \label{Atransform}
\end{equation}
where ${A}:\underline{g}\mapsto (c_{0,0}(\underline{g}),\hdots,c_{N,N}(\underline{g}))$, and $[\widetilde{{Q}}]_{k\ell} = \Lambda Y_\ell(x_k)$, where $Y_\ell$ is the $\ell'$th spherical harmonic in the natural order. We can think of ${A}$ as the matrix that takes a point-cloud representation of a function on $\partial \Omega$ and gives the spherical harmonic representation. }

Similarly to \cite{Knudsen2009a}, an approximation of the operator norm then takes the form
\begin{equation}\label{normapprox}
	\|\Lambda\|_{s,-s} \simeq \sup_f \frac{\|{{\mathcal{Q}}}f\|_{\mathbb{C}^{(N+1)^2}}}{\|f\|_{\mathbb{C}^{(N+1)^2}}} = \|{{{\mathcal{Q}}}}\|_{N},
\end{equation}
where $[{{{\mathcal{Q}}}}]_{ij} = \langle \Lambda \phi^s_{n,m}, \phi^{{-}s}_{n',m'} \rangle_{-s}$ with $i = n'^2+n'+m'+1$ and $j = n^2+n+m+1$. We may approximate
\begin{align}
	\langle \Lambda \phi^s_{n,m}, \phi^{{-}s}_{n',m'} \rangle_{-s} &= w_{-s}(n)w_{-s}(n')\langle \Lambda Y_n^m, Y_{n'}^{m'}\rangle, \\
	&\simeq w_{-s}(n)w_{-s}(n') {c_{n',m'}(\underline{\Lambda Y_{n}^m})}.  \label{innerapprox}
\end{align}
{With $\mathcal{B}$ we denote the map that takes the matrix ${Q}$ and gives the approximation of ${{\mathcal{Q}}}$ defined by \eqref{innerapprox}}. For $\Lambda = \Lambda_\gamma$ we denote the approximation \eqref{lambdaapprox}, ${Q}_\gamma$.

From \eqref{lambdaapprox} {it} is clear that to represent $\Lambda_\gamma$ we need only to compute $(\Lambda_\gamma Y_n^m)(x_k)$ in the quadrature points $x_k$. In this paper we compute $(\Lambda_\gamma Y_n^m)(x_k)$ efficiently by the boundary integral approach for piecewise constant conductivities given in \cite{delbary2014a}, an approach which despite the lack of reconstruction theory has shown to perform well.
\subsection{Noise model}
We simulate a perturbation of the Dirichlet-to-Neumann map by adding Gaussian noise to ${Q}_\gamma$. We let
\begin{equation}
	{Q}_\gamma^\varepsilon = {Q}_\gamma+\delta {E}, \label{noisemodel}
\end{equation}
where $\delta >0$ and the elements of the $2(N+1)^2\times 2(N+1)^2$ matrix ${E}$ are independent Gaussian random variables with zero mean and unit variance. We modify ${E}$ such that $\mathcal{B}{E}$ has a first row and column as zeros, such that we may consider $\mathcal{B}{E}$ as an approximation of a linear and bounded operator $\mathcal{E}\in Y$. Furthermore, we approximate $\|\mathcal{E}\|_Y$ using \eqref{normapprox} and \eqref{innerapprox} and note we can specify an absolute level of noise $\|\mathcal{E}\|_Y \approx \varepsilon$ by choosing $\delta$ appropriately. The relative noise level is then
\begin{equation}
	\delta \frac{\|\mathcal{E}\|_Y}{\|\Lambda_\gamma\|_Y} \approx \delta \frac{\|{\mathcal{B}}{E}\|_{N}}{\|{\mathcal{B}}{Q}_\gamma\|_{N}}.
\end{equation}
{Note the noise model in \cite{delbary2014a} scales each element of ${E}$ with the corresponding element of ${Q}_\gamma$.  Noise models for electrode data simulation typically takes the form
$$V_j^\varepsilon = V_j+\delta_j E_j,$$
as in \cite{hamilton2020a}, where $V_j$ is the voltage vector corresponding to the $j$'th current pattern, $\delta_j>0$ is a scaling parameter dependent on $V_j$ and $E_j$ is a Gaussian vector independent of $E_{j'}$ for $j\neq j'$. For our case such a noise model corresponds best to adding to $\widetilde{{Q}}_\gamma$ in \eqref{Atransform} a matrix $\widetilde{{E}}$ whose columns are $\delta_jE_j$. One may check by vectorizing ${A}^T \widetilde{{E}}^T$ that the corresponding ${E}$ of \eqref{noisemodel} consists of independent and identically distributed Gaussian vectors as rows. However, the elements of each row are now correlated with covariance matrix ${A}^T\mathrm{diag}(\delta){A}$.}

{Finally, we define the signal-to-noise ratio as
\begin{equation}
	\mathrm{SNR} = \frac{1}{(N+1)^2}\sum_{n=0}^N\sum_{m=-n}^n \frac{\|{Q}_\gamma \underline{Y_n^m}\|_{\mathbb{C}^{2(N+1)^2}}}{\delta\|{E} \underline{Y_n^m}\|_{\mathbb{C}^{2(N+1)^2}}}.
\end{equation}
}

\subsection{Solving the boundary integral equations}
Following \cite{delbary2014a} we discretize the perturbed boundary integral equations \eqref{bienoisy} by
\begin{equation}\label{discbie}
\left[I+(\mathcal{S}_{0} L_{N}+\mathcal{H}_{\zeta}^{N})(\Lambda_{\gamma}^\varepsilon-\Lambda_{1}) L_{N} \right] ((\psi_{\zeta}^{N})^\varepsilon|_{\partial \Omega})=e_{\zeta}|_{\partial \Omega},
\end{equation}
where $\mathcal{H}_{\zeta}^{N}$ is the approximation of {the integral operator $\mathcal{S}_\zeta-\mathcal{S}_0$} using the Gauss-Legendre quadrature rule of order $N+1$ on the unit sphere in the aforementioned quadrature points $\{x_k\}_{k=1}^{2(N+1)^2}$. We find the following result regarding the convergence of the perturbed solutions $(\psi_\zeta^N)^\varepsilon$ of \eqref{discbie} analogously to \cite{delbary2012a,delbary2014a}.
\begin{theorem}
	Suppose $D<|\zeta(\xi)|<-\frac{1}{6}\log\varepsilon_2$ and $\mathcal{E}$ is a linear bounded operator from $H^s(\partial \Omega)$ to $H^t(\partial \Omega)$ for all $s\geq 1/2$ and $t>s$.
	Then for all $s>3/2$, there exists $N_0 \in \mathbb{N}$ such that for all $N\geq N_0$ the operator $I+(\mathcal{S}_{0} L_{N}+\mathcal{H}_{\zeta}^{N})(\Lambda_{\gamma}^\varepsilon-\Lambda_{1}) L_{N} $ is invertible in $H^s(\partial \Omega)$. Furthermore we have,
	\begin{equation}
		\|(\psi_{\zeta}^{N})^\varepsilon - \psi_\zeta^\varepsilon\|_{H^s(\partial \Omega)}\leq \frac{C}{N^{s-3/2}}\|e_\zeta\|_{H^s(\partial\Omega)}.
	\end{equation}
\end{theorem}
\begin{proof}
The result follows from a Neumann series argument as in Lemma 3.1 and Theorem 3.2 of \cite{delbary2014a} as for $D<|\zeta(\xi)|<-\frac{1}{6}\log\varepsilon_2$, there exists a bounded inverse $(B_\zeta^\varepsilon)^{-1}$ by Lemma \ref{lemma3}.
\end{proof}
This result ensures that the solutions of the discretized perturbed boundary integral equations are unique and converge to the solutions of \eqref{bienoisy}.
\subsection{Choice of $|\zeta(\xi)|$ and truncation radius}
It is clear from Method \ref{Method2} that we should set $|\zeta(\xi)|=M^{p}$ for some exponent $p>3/2$. Due to the high sensitivity of the CGO solutions with respect to $|\zeta(\xi)|$, we may choose $|\zeta(\xi)|$ differently in practice, although we will not necessarily have a regularization strategy in theory. One idea of \cite{delbary2012a} is to set $|\zeta(\xi)|$ minimal in the admissible set \eqref{zetaxieq}, that is
\begin{equation}\label{fix}
	|\zeta(\xi)| = \frac{M}{\sqrt{2}}.
\end{equation}
A different idea is to choose $|\zeta(\xi)|$ independently for each $\xi$ such that $|\zeta(\xi)|$ is minimal with $|\zeta(\xi)| = \frac{|\xi|}{\sqrt{2}}$. We take the critical choice $|\zeta(\xi)|=K_1M^{3/2}$ for some constant $0<K_1<1$ to maintain the smallest $|\zeta|$ within the boundaries of the theory.

In practice we compute $\mathbf{t}_{M(\varepsilon)}(\xi,\zeta(\xi))$ in a $\xi$-grid of points $|\xi|\leq M$ as in \cite{delbary2014a}. The Shannon sampling theorem ensures we can recover uniquely the inverse Fourier transform if we sample densely enough. We use the discrete Fourier transform in equidistant $\xi$- and $x$-grids in three dimensions.
\begin{equation}
	\xi_k^j = -M+k\frac{2M}{K-1} \quad \text{ and } \quad x_n^j = -x_{\mathrm{max}}+n\frac{2x_{\mathrm{max}}}{K-1},
\end{equation}
for $n,k=0,\hdots,K-1$, $j=1,2,3$ and some $x_{\max}$ determined by $K$ and $M$. Indeed the discrete Fourier transform requires
\begin{equation}\label{def:K}
	M= \frac{\pi (K-1)^2}{2Kx_{\mathrm{max}}}.
\end{equation}
to recover $q^{\varepsilon}(x_n^j)$ for all $n=0,\hdots,K-1$, $j=1,2,3$. In practical applications, we do not know the noise level, in which case we choose $M$ and $K$ and consequently determine $x_{\mathrm{max}}$. Then we recover $q^{\varepsilon}$ in an appropriate finite element mesh of the unit ball using trilinear interpolation. The discrete Fourier transform is computed efficiently with the use of FFT \cite{frigo2005design} with complexity $\mathcal{O}(K^3\log{K^3})$.

The problem of finding the optimal truncation radius given noisy data $\Lambda_\gamma^\varepsilon$ is largely open and is related to the problem of systematically choosing a regularization parameter of regularized reconstruction for an inverse problem. In this paper, we choose the truncation radius by inspection for the simulated data. For further details on the implementation of the reconstruction algorithm we refer to \cite{delbary2012a,delbary2014a}.
\section{Numerical results}\label{sec:5}
We test Method \ref{method:2} as a regularization strategy. We are interested in whether the reconstruction converges to the true conductivity distribution as the noise level goes to zero, and likewise as the regularization parameter $\alpha$  goes to zero for a non-noisy Dirichlet-to-Neumann map. To this end, we simulate a Dirichlet-to-Neumann map for a well-known phantom.
\subsection{Test phantom}
The piecewise constant heart-lungs phantom consists of two spheroidal inclusions and a ball inclusion embedded in the unit sphere with a background conductivity of $1$. The phantom is summarized in Table \ref{hl-phantom}. We compute and represent the Dirichlet-to-Neumann map and noisy counterparts as described in Section \ref{sec41}. In particular, the forward map is computed using $2(N+1)^2$ boundary points on the unit sphere and using maximal degree $N$ of spherical harmonics with $N=25$.
\begin{table}[ht!]
\centering
\resizebox{\textwidth}{!}{%
\begin{tabular}{@{}lllll@{}}
\toprule
Inclusion & Center            & Radii       & Axes & Conductivity \\ \midrule
Ball      & $(-0.09,-0.55,0)$ & $r = 0.273$ &      & 2            \\ \midrule
Left spheroid &
  $0.55(-\sin(\frac{5\pi}{12}), \cos(\frac{5\pi}{12}), 0)$ &
  \begin{tabular}[c]{@{}l@{}}$r_1 = 0.468$,\\ $r_2 = 0.234$,\\ $r_3 = 0.234$\end{tabular} &
  \begin{tabular}[c]{@{}l@{}}$(\cos(\frac{5\pi}{12}),\sin(\frac{5\pi}{12}),0)$, \\ $(-\sin(\frac{5\pi}{12}),\cos\frac{5\pi}{12}),0), $\\ $(0,0,1)$\end{tabular} &
  0.5 \\ \midrule
Right spheroid &
  $0.45(\sin(\frac{5\pi}{12}), \cos(\frac{5\pi}{12}), 0)$ &
  \begin{tabular}[c]{@{}l@{}}$r_1 = 0.546$,\\ $r_2 = 0.273$,\\ $r_3 = 0.273$\end{tabular} &
  \begin{tabular}[c]{@{}l@{}}$(\cos(\frac{5\pi}{12}),-\sin(\frac{5\pi}{12}),0)$, \\ $(\sin(\frac{5\pi}{12}), \cos(\frac{5\pi}{12}), 0)$, \\ $(0,0,1)$\end{tabular} &
  0.5 \\ \bottomrule
\end{tabular}%
}
\caption{Summary of piecewise constant heart-lungs phantom consisting of three inclusions}
\label{hl-phantom}
\end{table}
\begin{figure}[ht]
\centering
\makebox[\textwidth][c]{
  \subcaptionbox{}{\includegraphics[width=2.2in]{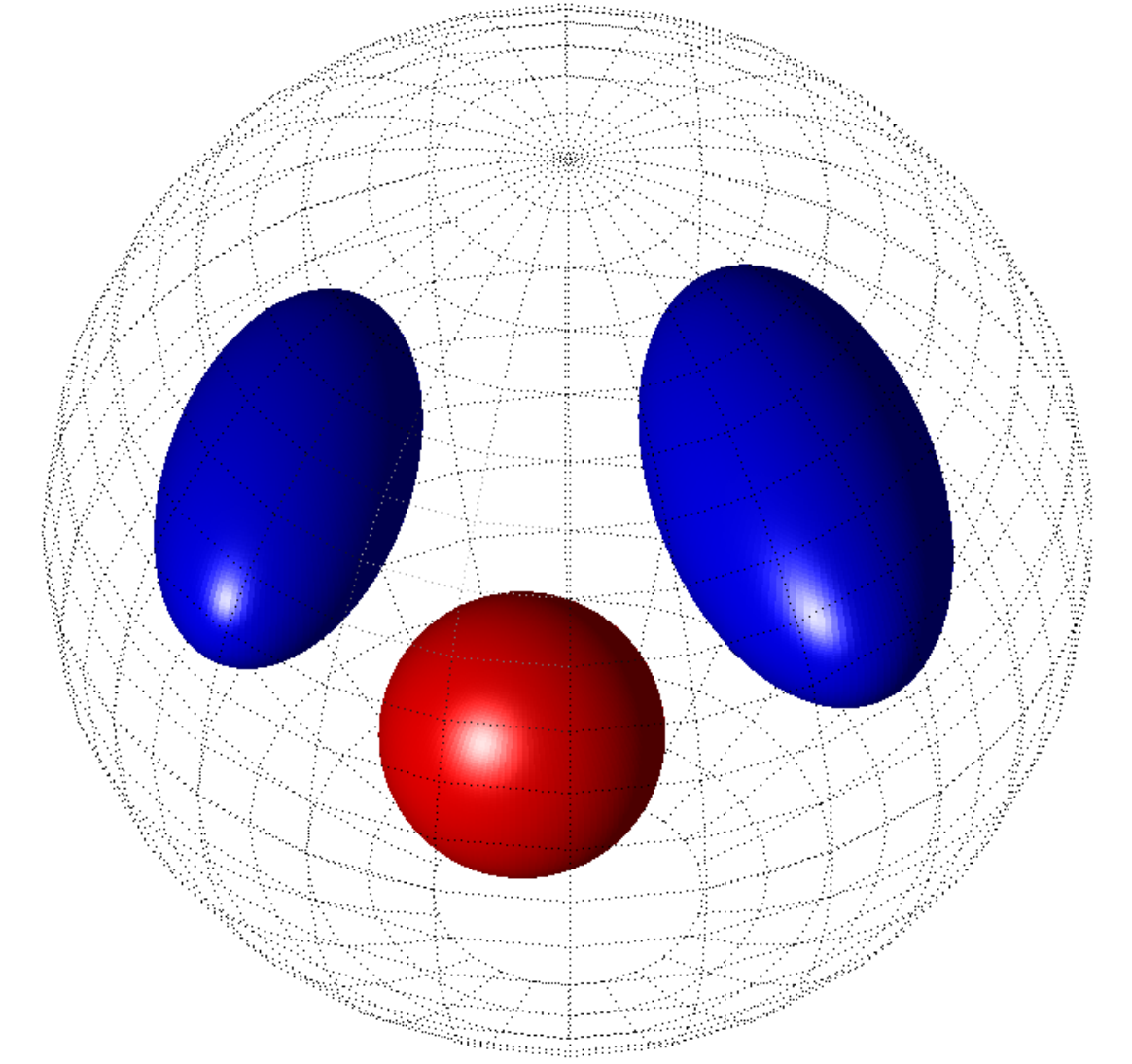}}\hspace{1em}%
  \subcaptionbox{}{\includegraphics[width=2.2in]{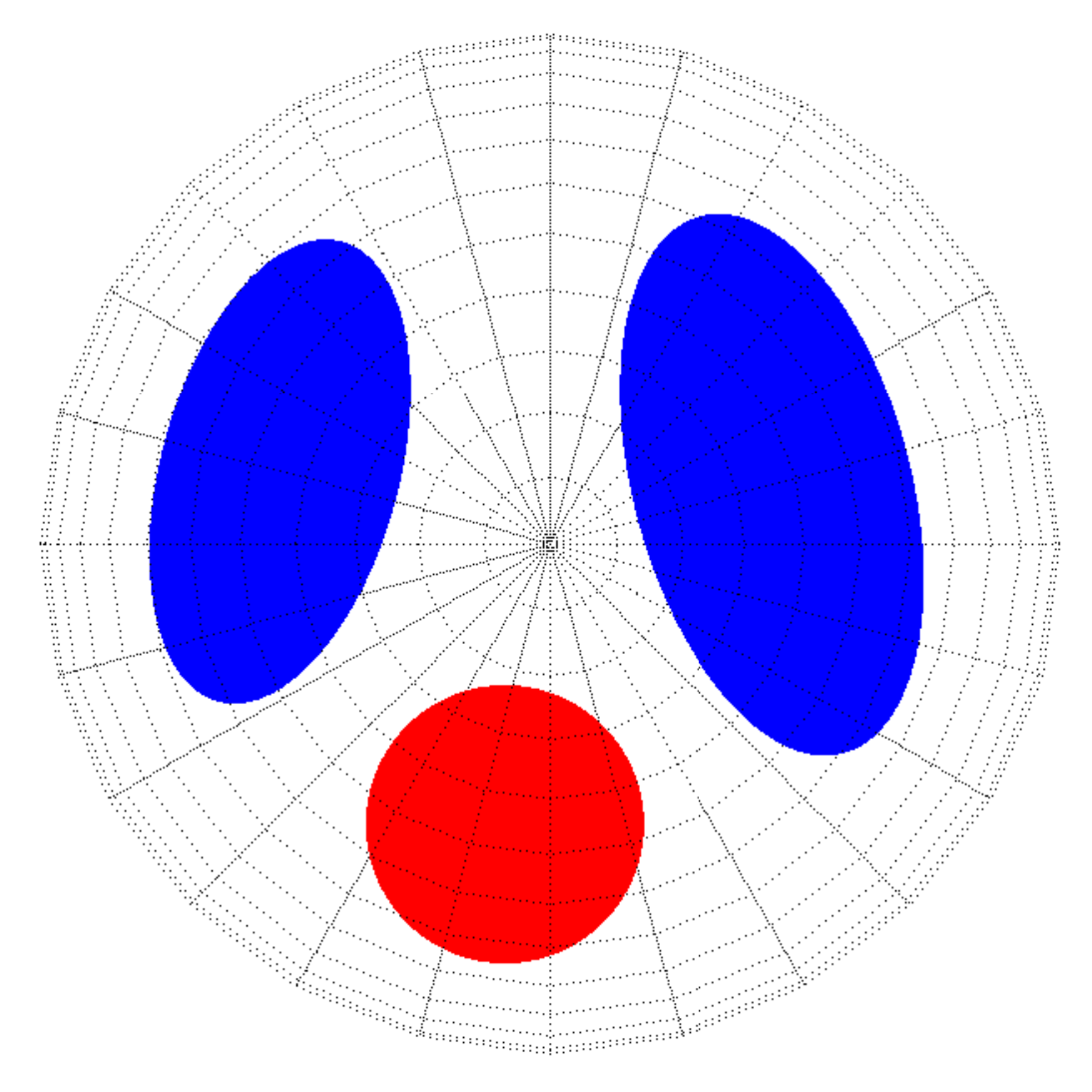}}
  }
  \caption{The piecewise constant heart-lungs phantom in a three-dimensional view (\textsc{a}), and in the planar cross section $x^3 = 0$ (\textsc{b}).}
  \label{fig:phantom1}
\end{figure}

\subsection{Regularization in practice}
We now consider the regularization strategy, Method \ref{method:2}, in practice. Alluding to  \eqref{reqweak}, we test the reconstruction algorithm by keeping the test data fixed and varying the regularization parameter.
\begin{figure}
	\includegraphics[width = 0.92\textwidth]{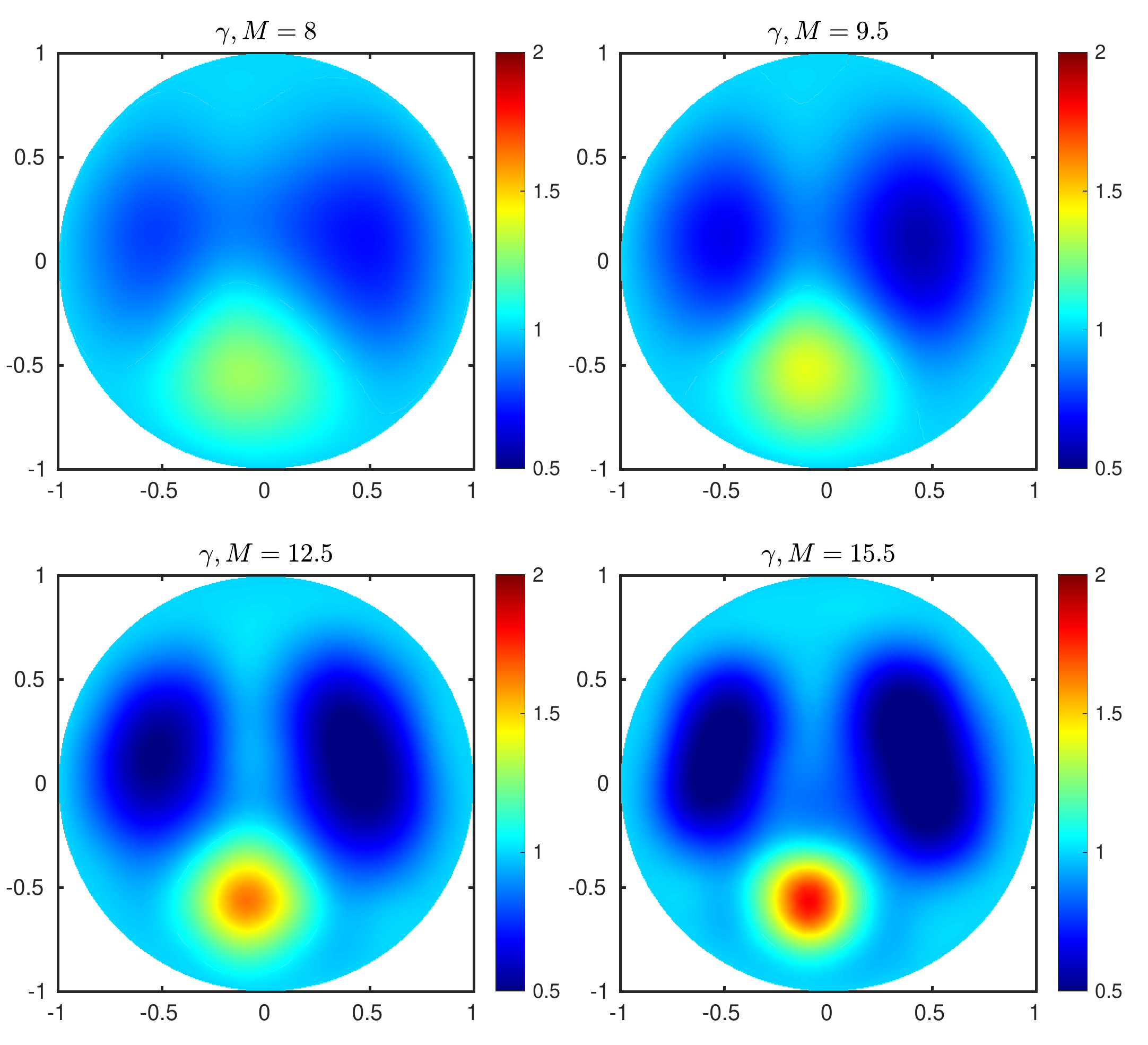}
	\caption{Cross sections $(x^3=0)$ of reconstructions using the regularized reconstruction algorithm with different choices of truncation radius $M$, $K=12$ and $|\zeta(\xi)|=\frac{1}{4}M^{3/2}$. There is no added noise.}
	\label{fig:reg1}\vspace*{-10pt}
\end{figure}
In Figure \ref{fig:reg1}, we see cross-sectional plots of reconstructed conductivities for different truncation radii $M=\alpha^{-1}$. We use $|\zeta(\xi)|=\frac{1}{4} M^{3/2}$ as the critical choice such that $\zeta(\xi)\in \mathcal{V}_\xi$ for $M\geq 8$, and use the accurate Dirichlet-to-Neumann map with no added noise. The figure shows increasing accuracy and contrast for increasing truncation radii. Similar to the findings of \cite{delbary2014a}, we experience failing reconstructions for large enough truncation radii as the frequency data is dominated by exponentially amplified noise inherent to the finite-precision representation of $\Lambda_\gamma$. This happens since there is noise present in the representation of the Dirichlet-to-Neumann map, no matter how accurately it represents the true infinite-precision data. We see the effect of truncation in practice: low resolution, smaller dynamical range and more smoothness caused by the missing high frequency data. Though not immediately clear from this figure, the reconstructions slightly overshoot  the conductivity of the resistive spheroidal inclusions with conductivities as small as 0.38. In addition, the reconstruction algorithm seems to work well in practice on piecewise constant conductivity distributions.

In Figure \ref{fig:reg2}, we see cross-sectional plots of reconstructed conductivities using Dirichlet-to-Neumann maps with added noise and for fixed $|\zeta(\xi)|= \frac{1}{3\sqrt{2}}M^{3/2}$. {Here, $K_1$ is chosen such that $\zeta(\xi)$ is small and admissible for $M\geq 9$.} The truncation radii are chosen optimally by visual inspection. The figure shows reconstructions in the presence of noise of levels ranging from $\varepsilon=10^{-6}$ to $\varepsilon=10^{-3}$ in the Dirichlet-to-Neumann map. We see improving quality of reconstruction as the noise level decreases in accordance with Definition \ref{def:reg2}. Beyond noise levels of $10^{-3}$, reconstruction is still feasible without the corruption of unstable noise, although, they need heavy regularization and start to lack visible features of the phantom. In Figure \ref{fig:noisyreg}, we see the conductivity reconstruction using noisy data with $\varepsilon=10^{-{2}}$ corresponding to approximately $1\%$ relative noise. The resistive spheroidal inclusions start to connect {and the conductive spherical inclusion is not as accurately placed. The remaining intensity in the signal compared to the case $M=9.7$ in Figure \ref{fig:noisyreg} could suggest that additional regularization is needed.}
\begin{figure}[!htbp]
	\includegraphics[width = 0.92\textwidth]{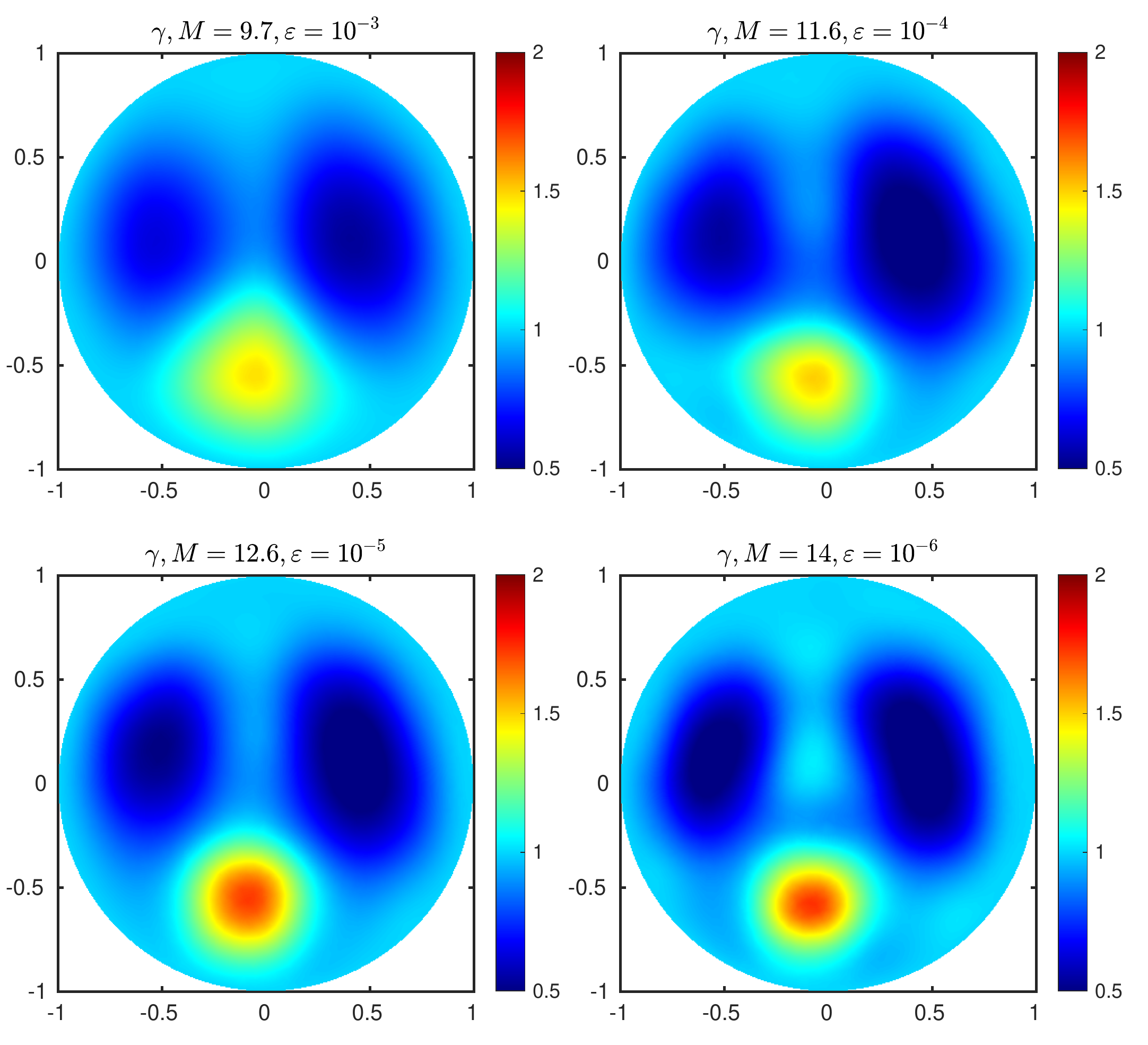}
	\caption{Cross sections $(x^3=0)$ of reconstructions using the regularized reconstruction algorithm on noisy Dirichlet-to-Neumann maps. {The noise levels correspond to relative noise levels $\varepsilon \approx 0.1\%$ with $\mathrm{SNR}=12\cdot 10^3$ (top left), $\varepsilon \approx 0.01\%$ with $\mathrm{SNR}=123\cdot 10^3$ (top right), $\varepsilon \approx 0.001\%$ with $\mathrm{SNR}=1172\cdot 10^3$ (bottom left) and $\varepsilon \approx 0.0001\%$ with $\mathrm{SNR}=11299\cdot 10^3$ (bottom right)}. The parameters used are $K=11$ and $|\zeta(\xi)|=\frac{1}{3\sqrt{2}}M^{3/2}$.}
	\label{fig:reg2}
\end{figure}

\begin{figure}
\centering
\makebox[\textwidth][c]{
  \subcaptionbox{}{\includegraphics[width=2.3in]{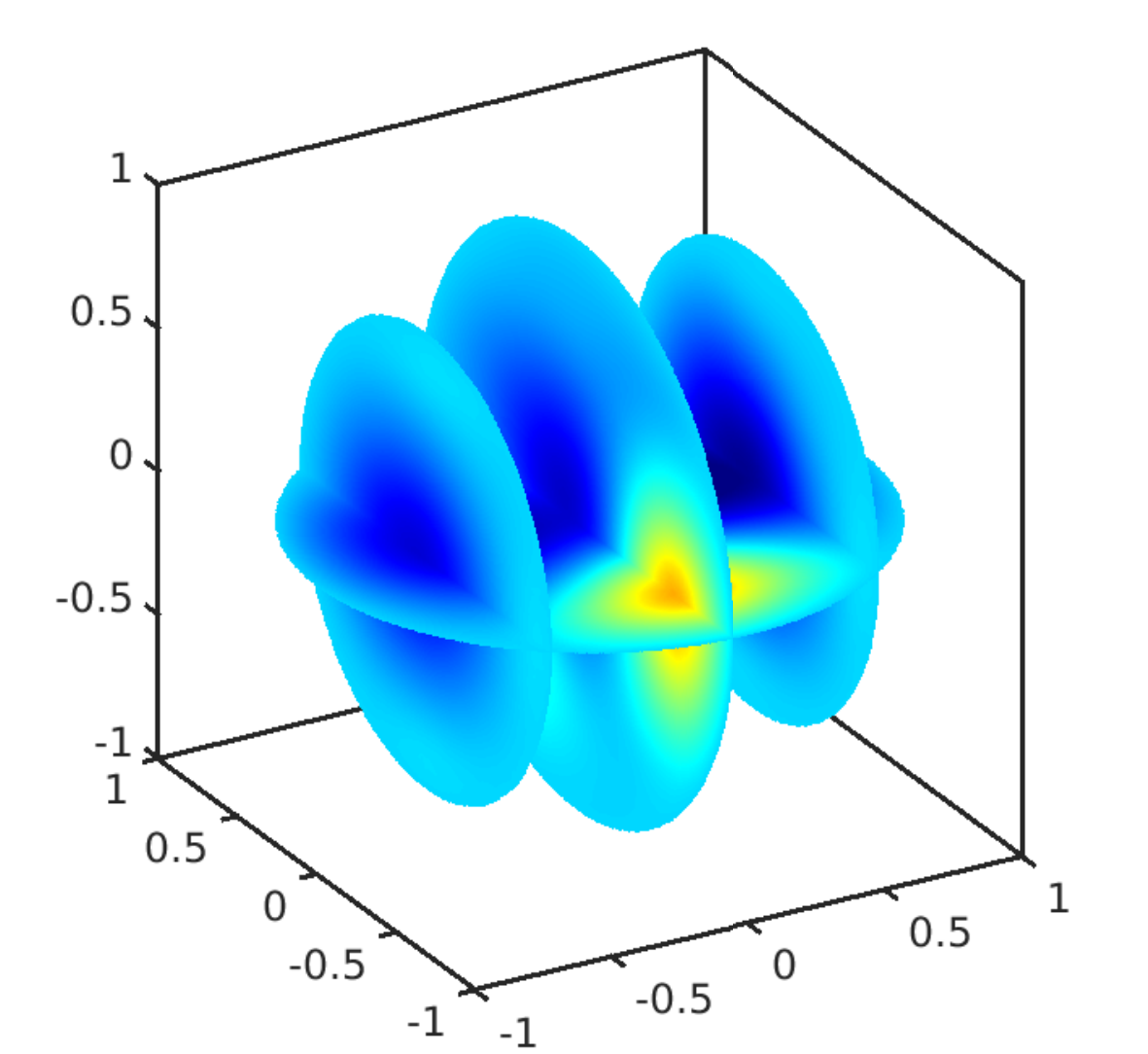}}\hspace{-1em} \label{fig:partA}
  \subcaptionbox{}{\includegraphics[width=2.3in]{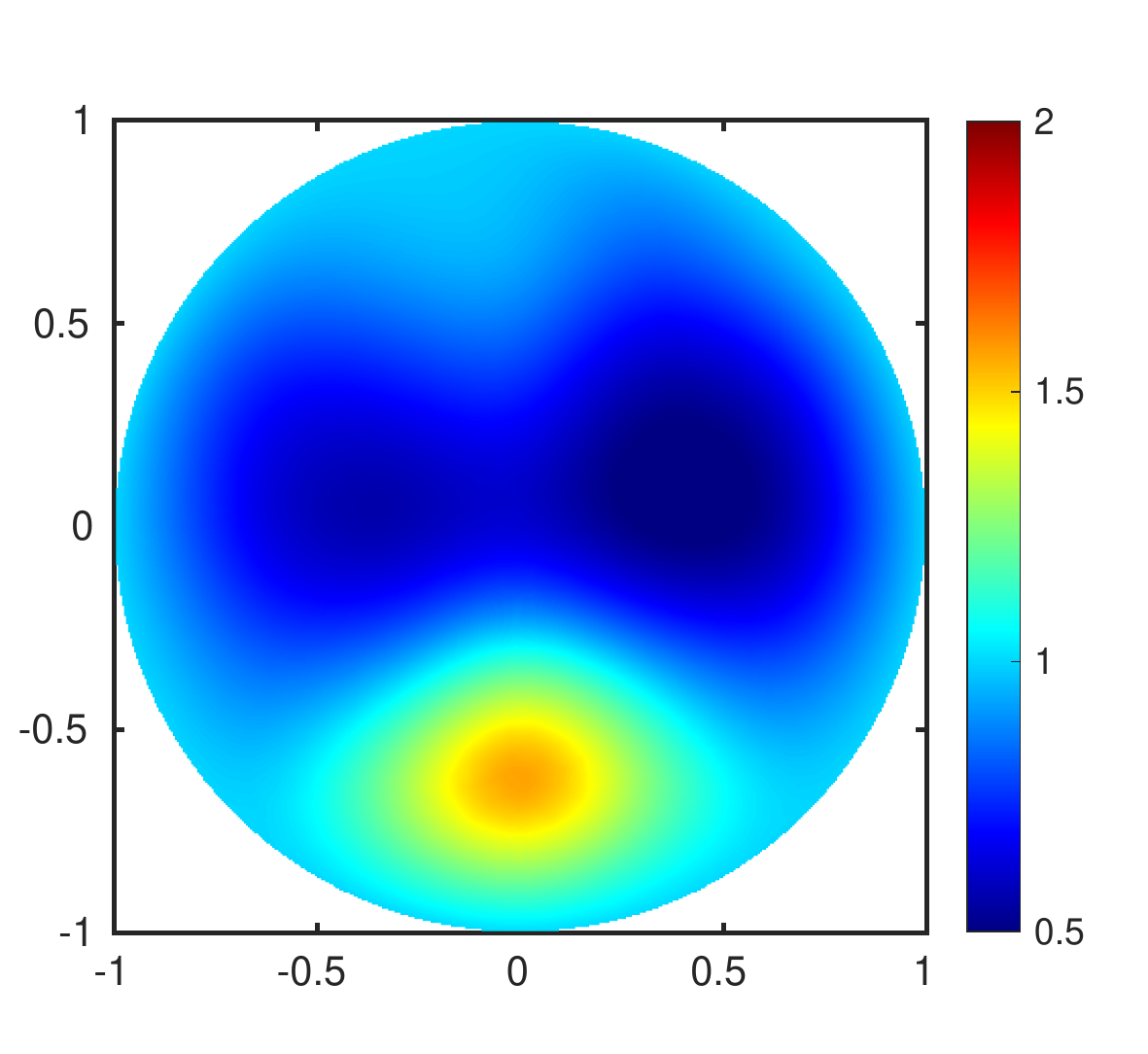}}\label{fig:partB}
  }
  \caption{Regularized reconstruction using noisy Dirichlet-to-Neumann maps with $\varepsilon = 10^{-2}$, which corresponds to approximately $1\%$ relative noise {and $\mathrm{SNR}=1.17\cdot 10^3$}. Plot (\textsc{a}) shows the cross sections $x^3=0$, $x^2 = -0.6$, $x^2 = {-0.05}$ and $x^2 = 0.6$, whereas plot (\textsc{b}) shows the plane corresponding to  $x^3 = 0$. The parameters used are $M=9$, $K=11$ and $|\zeta(\xi)|=\frac{1}{3\sqrt{2}}M^{3/2}$.}
  \label{fig:noisyreg}
\end{figure}

\begin{figure}
	\includegraphics[width = 0.8\textwidth]{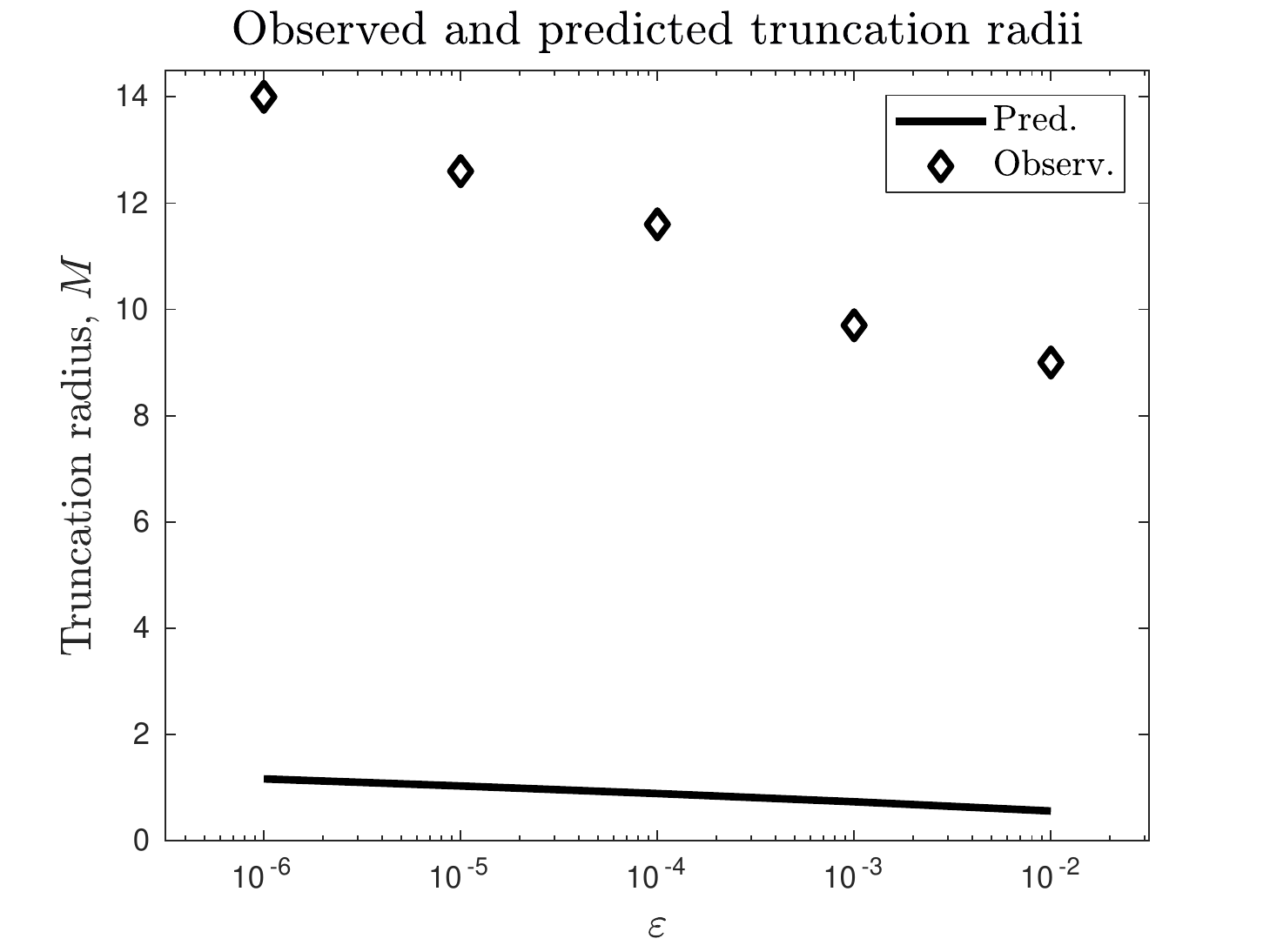}
	\caption{The truncation radii as predicted by theory $M =(-1/11\log(\varepsilon))^{-1/p}$ for $p=3/2$, and the chosen truncation radii for the noisy reconstructions of Figure \ref{fig:reg2} and \ref{fig:noisyreg}. }
	\label{fig:reg3}
\end{figure}

The truncation radii of reconstructions in Figure \ref{fig:reg2} and \ref{fig:noisyreg} chosen by visual inspection are plotted and compared to the theoretically predicted truncation radius in Figure \ref{fig:reg3}. This comparison suggests the prediction is somewhat  pessimistic and that the practical algorithm allows for lighter regularization in comparison to what the theoretical estimates portend. {However, the prediction and practical reconstructions are not directly comparable, since we should pick $|\zeta(\xi)|=K_1M^p$ with $p$ strictly larger than $3/2$ according to theory}. Finally, we note the noise model utilized by \cite{delbary2014a} and \cite{hamilton2020a} give somewhat different results compared to our {unnormalized perturbation}.
{The results also raise the question of how practical the reconstruction method is for more realistic data. Had we decreased the resolution of the basis of spherical harmonics to which voltages and currents are projected, the approximation error of highly oscillatory functions would increase. In this case we can expect to pick the truncation radius smaller to get a stable reconstruction. Investigating the reconstruction method for electrode data is subject to further study and is related to  \cite{isaacson2004} for the two-dimensional D-bar method and \cite{hamilton2020a} for the three-dimensional so-called $\mathbf{t}^{\mathrm{exp}}$ approximation. Possible improvements to the truncation strategy\break include extending the support of $\mathbf{t}$ with prior information using the forward map as in \cite{MR3554880}. In addition, one could experiment with a truncation by thresholding as in \cite{MR3626801}.}

\section{Conclusions}
In this paper we provide and investigate a regularization strategy for the Calder\'on problem in three dimensions. The main result of the paper is Theorem \ref{maintheorem}, which shows that the algorithm defined by Method \ref{Method2} yields reconstructions approximating the true conductivity, when using data corrupted by a sufficiently small perturbation. The proof relies on a gap of the magnitude of the complex frequency in which the existence of unique CGO solutions is guaranteed and the noise level allows a stable and unique solution to the boundary integral equation. The reconstructions from this strategy are regular as a result of the spectral filtering. Numerical results show the regularizing behavior of the reconstruction algorithm in practice and suggests one can utilize higher frequency information in the data than suggested by the theory. The reconstructions of piecewise constant conductivity data show promise even in the case of $1\%$ relative noise.

%For acknowledgements section, please don't number the section, please begin it with \section*{Acknowledgements}
\section*{Acknowledgments}
AKR and KK were supported by The Villum Foundation (grant no. 25893).

% You may incorporate your references as follows in your main tex file.
% Using BibTex is not recommended but can be handled.

\providecommand{\href}[2]{#2}
\providecommand{\arxiv}[1]{\href{http://arxiv.org/abs/#1}{arXiv:#1}}
\providecommand{\url}[1]{\texttt{#1}}
\providecommand{\urlprefix}{URL }

\medskip
% The data information below will be filled by AIMS editorial staff
%Received June 2021; revised November 2021; early access January 2022.
\medskip

\end{document}